\documentclass[final,5p,times,twocolumn]{elsarticle}

\usepackage[english]{babel}
\usepackage[usenames,dvipsnames,svgnames,table]{xcolor}
\usepackage{latexsym,amsmath,amssymb,amsfonts,graphicx,graphics,epsfig,epsf,lscape,subcaption}

\usepackage[utf8x]{inputenc}
\usepackage{amsmath}
\usepackage{amsfonts}
\usepackage{graphicx}
\graphicspath{{./fig/}}

\usepackage{float}
\usepackage[section]{placeins}
\usepackage{bm}

\usepackage{wrapfig}
\usepackage{amssymb}
\DeclareMathOperator*{\argmax}{arg\,max}
\DeclareMathOperator*{\argmin}{arg\,min}
\usepackage{dsfont}
\usepackage{pgf}
\usepackage{tikz}
\usepackage{tkz-graph}
\usetikzlibrary{arrows,automata,shapes,shapes.multipart,positioning, fit, shapes.geometric}
\usetikzlibrary{arrows,automata}

\usepackage{amsthm}

\usepackage[vlined,ruled]{algorithm2e}

\usepackage{enumerate}

\tikzstyle{SimpleNetwork}=[draw,circle,minimum width=6 pt]

\def \etal {\emph{et al.}}

\theoremstyle{definition}
\newtheorem{theorem}{Theorem}

\newtheorem{proposition}{Proposition}
\newtheorem{corollary}{Corollary}
\newtheorem{definition}{Definition}
\newtheorem{example}{Example}
\newtheorem{remark}{Remark}

\usepackage[font=footnotesize]{caption}

\def \bs {\boldsymbol}
\def \mc {\mathcal}

\def\prob{\mathbb{P}}

\def\expt{\mathbb{E}}
\def\real{\mathbb{R}}

\def\naturals{\mathbb{N}}
\newcommand{\until}[1]{\{1,\dots, #1\}}

\newcommand{\supscr}[2]{#1^{\textup{#2}}}
\newcommand{\setdef}[2]{\{#1 \; | \; #2\}}
\newcommand{\seqdef}[2]{\{#1\}_{#2}}

\newcommand{\map}[3]{#1: #2 \rightarrow #3}

\newcommand{\slfrac}[2]{\left.#1\middle/#2\right.}
\newcommand{\lnn}[1]{%
	\ln  \left(#1\right)%
}
\newcommand{\expp}[1]{%
	\exp \! \left(#1\right)%
}

\newcommand\oprocendsymbol{\hbox{$\square$}}
\newcommand\oprocend{\relax\ifmmode\else\unskip\hfill\fi\oprocendsymbol}

\newcommand\bit[1]{\textit{\textbf{#1}}}

\makeatletter
\newcommand{\vast}{\bBigg@{4}}
\newcommand{\Vast}{\bBigg@{5}}
\makeatother

\begin{document}

\begin{frontmatter}

\title{Distributed Cooperative Decision Making in Multi-agent Multi-armed Bandits \tnoteref{lab}}
\tnotetext[lab]{This research has been supported in part by ONR grants  N00014-14-1-0635, N00014-19-1-2556,  ARO grants W911NF-14-1-0431, W911NF-18-1-0325 and by the Department of Defense (DoD) through the National Defense Science \& Engineering Graduate Fellowship (NDSEG) Program.}

\author[label1]{Peter Landgren}
\author[label2]{ Vaibhav Srivastava}
\author[label1]{Naomi Ehrich Leonard}

\address[label1]{Department of Mechanical \& Aerospace Engineering, Princeton University, New Jersey, USA, \emph{\tt{ \{landgren, naomi\}@princeton.edu}}}
\address[label2]{Department of Electrical and Computer Engineering, Michigan State University, East Lansing, MI, USA, \tt{vaibhav@egr.msu.edu}}
\begin{abstract}
	We study a distributed decision-making problem in which multiple agents face the same multi-armed bandit (MAB), and each agent makes sequential choices among arms to maximize its own individual reward. The agents cooperate by sharing their estimates over a fixed communication graph. We consider an unconstrained reward model in which two or more agents can choose the same arm and collect independent rewards. And we consider a constrained reward model in which agents that choose the same arm at the same time receive no reward. We design a dynamic, consensus-based, distributed estimation algorithm for cooperative estimation of mean rewards at each arm. We leverage the estimates from this algorithm to develop two distributed algorithms: coop-UCB2 and coop-UCB2-selective-learning, for the unconstrained and constrained reward models, respectively. We show that both algorithms achieve group performance close to the performance of a centralized fusion center.  Further, we investigate the influence of the communication graph structure on performance. We propose a novel graph explore-exploit index that predicts the relative performance of groups in terms of the communication graph, and we propose a novel nodal explore-exploit centrality index that predicts the relative performance of agents in terms of the agent locations in the communication graph. 
\end{abstract}

\begin{keyword}
multi-armed bandits \sep multi-agent systems \sep distributed decision making \sep explore-exploit dilemma
\end{keyword}

\end{frontmatter}

\section{Introduction}
Many engineered and natural systems are faced with the challenge of decision making under uncertainty, in which an agent must make decisions among alternatives while still learning about those options.  Decision making under uncertainty inherently features the \emph{explore-exploit tradeoff}, where one must decide between selecting options with a high expected payoff (exploitation) and selecting options with less well-known but potentially better payoff (exploration).   Often systems feature multiple networked decision makers, where performance of the system may require \emph{cooperative} decision making, in which disparate and distributed elements of a group act collaboratively.

The explore-exploit tradeoff {can be} formally investigated {within the context of} the multi-armed bandit (MAB) problem.  In a stochastic MAB problem, an agent is presented with a set of arms (options), and each arm is represented by a stochastic reward with a mean that is unknown to the agent.  An agent's goal is to select arms sequentially in order to maximize its own cumulative expected reward over time.  Good performance in the MAB problem requires an agent to balance learning the mean reward of each arm (exploration) with choosing the arm with the highest estimated mean (exploitation).  

The explore-exploit tradeoff has been widely investigated using the MAB problem across a variety of scientific fields and has found diverse application in control and robotics~\cite{VS-PR-NEL:14, MYC-JL-FSH:13},  ecology~\cite{JRK-AK-PT:78, VS-PR-NEL:13}, and communications~\cite{anandkumar2011distributed}. The MAB problem, and particularly the classical single-agent variant, has been studied extensively (see~\cite{SB-NCB:12} for a survey).  In \cite{lai1985asymptotically}, Lai and Robbins established a limit on the expected performance of any optimal policy in a frequentist setting by proving a lower bound on the number of times an agent selects a sub-optimal arm.

To date most research on the MAB problem has focused on single-agent policies, but the rising importance of networked systems and large-scale information networks have motivated the investigation of the MAB problem with multiple agents. 
In this paper, we study two variants of the multi-agent MAB problem, in which each agent makes choices to maximize its own individual reward but cooperates by communicating its estimates across a network. 
The first variant assumes an {\em unconstrained reward} model, in which agents are not penalized if they choose the same arm at the same time.  
The second variant assumes a \emph{constrained reward}, in which agents that choose the same arm at the same time receive a reduced reward.  Consider agents in a remote setting choosing among communication channels to send data back to a base station.  The constrained reward model applies to the case in which data cannot be sent if agents choose the same channel.  The constrained reward model can also be used to prevent mobile agents from searching for resource in the same patch when there exist multiple resource-rich patches.


When a centralized fusion center that has access to all the information available to every agent decides which arms will be sampled by the agents, the agents are inherently coordinated and no two agents ever sample the same arm at the same time. In this setting, the above two variants become almost the same. Anantharam~\etal~\cite{VA-PV-JW:87} extended the classical single-agent MAB problem to the setting of such a fusion center and derived a fundamental lower bound on the performance of the fusion center. In this paper, we design {\em distributed} algorithms that yield group performance close to that of a centralized fusion center. 

Kolla~\etal~\cite{KollaJG16} and Landgren~\etal~\cite{PL-VS-NEL:18c} studied the multi-agent MAB problem under the unconstrained reward model. In their setup, each agent can share its actions and the associated rewards at each time with its neighbors in the communication graph. In this setting, group performance improves when each agent acts individually.  However, group performance might not be close to the performance of a centralized fusion center, especially for large sparse networks. Madhushani and Leonard~\cite{madhushani2019heterogeneous,madhushani2020observation} have extended this setting to examine dynamic interactions among agents  governed by a heterogeneous stochastic process and to design strategies that minimize sampling regret as well as communication costs.

Several researchers ~\cite{anandkumar2011distributed,liu2010distributed,kalathil2014decentralized,gai2014distributed,wei2018distributed} have studied the distributed multi-agent MAB problem under the constrained reward model. In these works, agents seek to converge on the set of best arms, but they do not explicitly communicate with one another.  
In~\cite{kalathil2014decentralized,gai2014distributed}, agents are ranked and they target the best arm  associated with their rank. Anandkumar~\etal~\cite{anandkumar2011distributed} also studied distributed policies for agents to learn their ranks while solving the multi-agent MAB problem. Bistritz and Leshem~\cite{bistritz2018distributed} studied the distributed multi-agent MAB problem under no communication among agents. Assuming no a priori ranking of agents, they developed a game-of-thrones algorithm, inspired by~\cite{marden2014achieving}, to enable coordination among agents.


Shahrampour~\etal~\cite{shahrampour_rakhlin_jadbabaie_2017}  studied a variant of the multi-agent MAB problem in which the reward associated with each arm may be different for every agent. The best arm is defined as the arm with the maximum average mean reward over all agents. Unlike in other multi-agent MAB setups, in which each agent makes a decision at each time, they consider a single group decision obtained using a majority rule on individual decisions.

 In early versions~\cite{arxiv:LandgrenSL15,DBLP:journals/corr/LandgrenSL16} of the present work, we studied distributed cooperative decision making in the multi-agent MAB problem with the unconstrained reward model. In comparison, this paper considers a broader class of reward distributions and studies both the unconstrained and constrained reward models. We present new detailed proofs that improve on 
 the preliminary versions. We also present a much broader exploration of the 
 influence of communication graph structure on individual and group decision-making performance. 

Mart{\'\i}nez-Rubio~\etal~\cite{martinez2019decentralized} extended our preliminary versions~\cite{arxiv:LandgrenSL15,DBLP:journals/corr/LandgrenSL16} 
in the context of the unconstrained reward model. Their work is complementary to the approach discussed here. A key difference between their algorithm and the algorithm discussed in this paper, is that our algorithm requires only the knowledge of total number of agents to tune the decision-making heuristic, while their algorithm requires the knowledge of the spectral gap of the communication graph.  They do not investigate the influence of the network graph on performance. 


In this paper, we study distributed cooperative decision making in the multi-agent MAB problem under both unconstrained and constrained reward models.  We use a set of running consensus algorithms for cooperative estimation of the mean reward at each arm over an undirected graph and develop algorithms for individual decision making based on these estimates for both reward models. We also derive measures of graph structure that are predictive of 
individual as well as group performance. The major contributions of the paper are as follows.

%

First, we employ and rigorously analyze running consensus algorithms for distributed cooperative estimation of mean reward at each arm, and we derive bounds on key quantities. 

Second, we propose and thoroughly analyze  the coop-UCB2 algorithm for the multi-agent MAB problem under the unconstrained reward model and sub-Gaussian reward distributions. 

Third, we propose and thoroughly analyze the coop-UCB2-selective-learning algorithm for the multi-agent MAB problem under the constrained reward model and sub-Gaussian reward distributions.  

Fourth, we utilize the derived bounds on the decision-making performance of the group to introduce a novel graph explore-exploit index that predicts the ordering of graphs in terms of group explore-exploit performance and a novel nodal explore-exploit centrality index as a function of an agent's location in a graph that predicts the ordering of agents in terms of individual explore-exploit performance. We illustrate the effectiveness of these indices with simulations.

The remainder of the paper is organized as follows.  In Section~\ref{sec:setup} we describe the multi-agent MAB problem studied in this paper and introduce some background material.  In Section~\ref{sec:coop-est} we present and analyze the cooperative estimation algorithm. We propose and analyze the coop-UCB2 algorithm in Section~\ref{sec:DistributedDecisionMaking} and the coop-UCB2-selective-learning algorithm in Section \ref{sec:Coop-ucb2-collisions}. We illustrate our analytic results with numerical examples in Section~\ref{NetworkPerformanceAnalysis}. We conclude in Section~\ref{FinalRemarks}.

\section{Problem Description} \label{sec:setup}
We consider a distributed multi-agent MAB problem in which $M$ agents make sequential choices among the same set of $N$ arms with the goal of maximizing their individual reward. The $M$ agents cooperate by sharing their estimates over a bi-directional communication network. The network is modeled by an undirected graph $\mathcal{G}$ in which each node represents a decision-making agent and edges represent the communication links between them~\cite{FB-JC-SM:09}. Let $A\in \real^{M\times M}$ be the adjacency matrix associated with $\mc G$ and  $L \in \real^{M\times M}$ the corresponding Laplacian matrix. We assume that the graph $\mc G$ is connected, i.e., there exists a path between every pair of nodes.

Let the reward associated with arm $i\in \until{N}$ be a stationary random variable with an unknown mean $m_i$.  Using its local information, each agent $k \in \until{M}$  selects arm $i^k(t)$ at time $t \in \until{T}$, where $T \in \naturals$ is the time horizon. 

We study two reward  models that determine how the reward associated with arm $i^k(t)$ is received by agent $k$. In the \emph{unconstrained reward model}, agent $k$ receives a reward equal to the realized value of the reward at arm $i^k(t)$, irrespective of the choices of the other agents. In the \emph{constrained reward model},  agent $k$ receives a reward equal to the realized value of the reward at arm $i^k(t)$, only if it  is the only agent to select arm $i^k(t)$ at time $t$; otherwise it receives no reward. 

The objective of the distributed cooperative multi-agent MAB problem is to maximize the expected cumulative group reward. This objective is equivalent to minimizing the \emph{expected cumulative group regret} defined by the difference between the best possible expected cumulative group reward and the achieved expected cumulative group reward. 

Let {$\seqdef{b^i}{i\in\until{N}}$} be the permuted sequence of arms such that  $m_{b^1} > m_{b^2} > \cdots > m_{b^N}$ {\footnote{{We rely on the assumption that $m_{b^i} \neq m_{b^j}$ for all $i,j$ for the constrained reward model; it can be relaxed for the unconstrained reward model.}}}.
Under the unconstrained reward model, the expected cumulative group regret is defined by
\begin{equation}
\supscr{R}{unc}_T = MT m_{b^1} - \sum_{t=1}^T \sum_{k=1}^M   m_{i^k(t)} = \sum_{i=1}^N \sum_{k=1}^M \Delta_i \expt[n_i^k(T)],
\label{Runc}
\end{equation}
where $n_i^k(T)$ is the total number of times arm $i$ is selected by  agent $k$  until time $T$ and $\Delta_i = m_{b^1}-m_i$. In the following, we use $b^1$ and $i^*$ interchangeably to denote the arm with the highest mean reward.  {Under the unconstrained reward model, the regret at time $t$ is minimized if every agent chooses arm $i^*$.}

Similarly, under the constrained reward model {and assuming $M \leq N$}, the expected cumulative group regret is defined by
\begin{equation}
\supscr{R}{con}_T = T \sum_{k=1}^M m_{b^k} - \sum_{t=1}^T \sum_{k=1}^M   m_{i^k(t)} \mathbb{I}_{i^k(t)}^k(t), 
\label{Rcon}
\end{equation}
where $\mathbb{I}_{i}^k(t) = 1$ if agent $k$ is the only agent to sample arm $i$ at time $t$, and $0$ otherwise. In the following, we  denote the set of {$j$ best arms by $\mc O_j^* =\{b^1, \ldots, b^j\}$.  Under the constrained reward model, the regret at time $t$ is minimized if each agent chooses a different arm in the set $\mc O_M^*$.}

Let $p_i$ be the probability distribution of the reward associated with arm $i$. For a centralized fusion center  that has access to information available to each agent, and under the unconstrained reward model, the  lower bound 
\begin{equation}
	\sum_{k=1}^M \expt[n_i^k (T)] \ge \left(\frac{1}{ \mathcal{D} (p_i || p_{i^*})}  +o(1) \right) \ln T \label{eqn:fusioncenterregret}
\end{equation}
holds asymptotically as $T \to +\infty$ for any suboptimal arm $i \ne i^*$~\cite{lai1985asymptotically,VA-PV-JW:87}. Here, 
$\mathcal{D} (p_i || p_{i^*})$ represents the Kullback-Leibler divergence between $p_i$ and {$p_{i^*}$}. Substituting, the lower bound on {$\sum_k \expt[n_i^k(T)]$} in~\eqref{eqn:fusioncenterregret} into the expression \eqref{Runc} for $\supscr{R}{unc}_T$ yields
\begin{equation}\label{eqn:regret-unc-lower}
\supscr{R}{unc}_T \ge  \sum_{i \ne i^*} \left(\frac{\Delta_i}{ \mathcal{D} (p_i || p_{i^*})}  +o(1) \right) \ln T. 
\end{equation}

For the constrained reward model, consider a centralized fusion center that has access to the information available to each agent and can assign the arm to be selected by each agent at each time. For such a fusion center, no two agents ever select the same arm at the same time, and, under the constrained reward model, the lower bound
\begin{equation}
\sum_{k=1}^M	 \expt[n_i^k(T)] \ge \left(\frac{1}{ \mathcal{D} (p_i || {p_{b^M})}}  +o(1) \right) \ln T \label{eqn:fusioncenterregret-cons}
\end{equation}
holds asymptotically as $T \to +\infty$ for any suboptimal arm {$i \notin \mc O_M^*$}~\cite{VA-PV-JW:87}. Thus, the asymptotic regret of the fusion center satisfies
\begin{equation}\label{eqn:regret-cons-lower}
\supscr{R}{con}_T \ge {\sum_{i\in \{b^{M+1}, \ldots, b^N\}}} \left(\frac{\Delta_i}{ \mathcal{D} (p_i || {p_{b^M})}}  +o(1) \right) \ln T.
\end{equation}
Note that the expected regret under the constrained reward model is higher if multiple agents select the same arm. Thus, the above lower bound holds even if agents themselves make arm selections instead of being assigned an arm by the fusion center. The situation in which multiple agents select the same arm is referred to as a \emph{collision}.

Our objective in this paper is to design a distributed cooperative algorithm estimating mean reward at each arm  and  a decision-making algorithm for each agent that yields expected cumulative group regret close to that of a {centralized fusion center}. 
We consider rewards drawn from a sub-Gaussian distribution. 

\begin{definition}[\bit{Sub-Gaussian random variable~\cite{boucheron_2016}}]\label{defn:sub-gaussian}
A real-valued random variable $X$, with $\expt[X] =m \in \real,$ is sub-Gaussian if 	
\begin{equation*}
	\phi_X (\beta) \leq m \beta +\frac{\sigma_g^2 \beta^2}{2},
\end{equation*}
where $\sigma_g \in \real_{>0}$, $\beta \in \real$, and $\map{\phi_X}{\real}{\real}$ is  the cumulant generating function of $X$ defined by
\[
\phi_X(\beta) = \lnn{\mathbb{E}[\expp{\beta X}]}.
\]
\end{definition}
Sub-Gaussian distributions include Bernoulli, uniform, and Gaussian distributions, and distributions with bounded support.

\section{Cooperative Estimation of Mean Rewards} \label{sec:coop-est}
In this section we study  cooperative estimation of mean rewards at each 
arm. We propose two running (dynamic) consensus algorithms~\cite{braca2008enforcing,olfati2004consensus} for each arm and analyze performance.

\subsection{Cooperative Estimation Algorithm}
For distributed cooperative estimation of the mean reward at each arm $i$, we propose two running consensus algorithms: (i) for estimation of total reward provided at arm $i$, and (ii) for estimation of the total number of times arm $i$ has been sampled. 

Let $\hat{s}_i^k(t)$  be agent $k$'s estimate of the total reward provided at arm $i$ until time $t$ per unit agent. Let $\hat{n}_i^k(t)$ be agent $k$'s estimate of the total number of times arm $i$ has been selected until time $t$ per unit agent. Recall that $i^k(t)$ is the arm sampled by agent $k$ at time $t$ and let $\xi_i^k(t) = \mathds{1} (i^k(t) =i)$.   $\mathds{1}(\cdot)$ is the indicator function, here equal to 1 if $i^k(t)= i$ and 0 otherwise. 
{For all $i$ and $k$,} we define $r_i^k(t)$ as the realized reward at arm $i$ for agent $k$ at time $t$, which is 
a random variable sampled from a sub-Gaussian distribution. The corresponding reward received by agent $k$ at time $t$ is $r^k(t) = r_i^k(t) \cdot \mathds{1} (i^k(t) =i)$.

The estimates $\hat{s}_i^k(t)$ and $\hat{n}_i^k(t)$ are updated using running consensus as follows
\begin{align} \label{nhatdefn}
\mathbf{\hat{n}}_i(t) &= {P \left( \mathbf{\hat{n}}_i(t-1) + \boldsymbol{\xi}_i(t)\right)}, \\
\text{and} \quad \mathbf{\hat{s}}_i(t) &= {P \left(\mathbf{\hat{s}}_i(t-1) + \mathbf{r}_i(t)\right)}, \label{shatdefn}
\end{align}
where $\mathbf{\hat{ n}}_i(t)$, $\mathbf{\hat{ s}}_i(t)$, $\boldsymbol{\xi}_i(t)$, and $\mathbf{r}_i(t)$ are vectors of $\hat n_i^k(t)$, $\hat s_i^k(t)$, $\xi_i^k(t)$, and $r_i^k(t) \cdot \mathds{1} (i^k(t) =i)$, $k\in \until{M}$, respectively; $P$ is a row stochastic matrix given by
\begin{equation} 
P= \mathcal{I}_M - \frac{\kappa}{d_{\text{max}}} L. \label{Pdefn}
\end{equation}
$\mathcal{I}_M$ is the identity matrix of order $M$, $\kappa \in (0,1]$ is a step size parameter \cite{olfati2004consensus}, $d_{\text{max}} =\max \setdef{\text{deg}(i)}{i \in \until{M}}$, and $\text{deg}(i)$ is the degree of node $i$. In the following, we assume without loss of generality, that the eigenvalues of $P$ are ordered such that $\lambda_1 = 1 > \lambda_2 \geq ... \geq \lambda_M > -1$. 

In the running consensus updates~\eqref{nhatdefn} and \eqref{shatdefn}, each agent $k$ collects information $\xi_i^k(t)$ and {$r^k(t)$} at time $t$, adds it to its current opinion, and then averages its updated opinion with the updated opinion of its neighbors.

 Using $\hat{s}_i^k(t) $ and  $\hat{n}_i^k(t)$, agent $k$ can calculate $\hat{\mu}_i^k(t)$, the estimated empirical mean of arm $i$ at time $t$ defined by
\begin{equation}
\hat{ \mu}_i^{k}(t) = \frac{\hat{s}_i^{k}(t)}{\hat{n}_i^{k}(t)}.  \label{eqnmean}
\end{equation}

\subsection{Analysis of the Cooperative Estimation Algorithm}
We now analyze the performance of the estimation algorithm defined by~\eqref{nhatdefn},~\eqref{shatdefn}~and~\eqref{eqnmean}. Let $n_i^{\text{cent}}(t) \equiv \frac{1}{M} \sum_{\tau=1}^t \mathbf{1}_M^\top \boldsymbol{\xi}_i(\tau) =\frac{1}{M} \sum_{k=1}^M n_i^k(t)$ be the total number of times arm $i$ has been selected per unit agent {until} time {$t$}, and let $s_i^{\text{cent}}(t) \equiv \frac{1}{M} \sum_{\tau=1}^t \boldsymbol{\xi}_i^\top(\tau) \mathbf{r}_i(\tau)$ be the
total reward provided at arm $i$ per unit agent {until} time $t$.  Let 
$\mathbf{u}_i$ be the eigenvector corresponding to $\lambda_i$, $u_i^d$  the $d$-th entry of $\mathbf{u}_i$. Note $\lambda_1=1$ and $\mathbf{u}_1 = \mathbf{1}_M/\sqrt{M}$. Let
\[
 \nu_{pj}^{\text{+sum}} = \sum_{d=1}^M  u_p^d u_j^d \mathds{1}(u_p^k u_j^k \geq 0),  
  \quad \nu_{pj}^{\text{-sum}} =  \sum_{d=1}^M u_p^d u_j^d \mathds{1}(u_p^k u_j^k \leq 0), 
\]
\begin{equation}
a_{pj}(k) =  \begin{cases}
\nu_{pj}^{\text{+sum}}u_p^k u_j^k, & 
\! \! \! \! \text{if } \lambda_p \lambda_j \geq 0 \; \&\; u_p^k u_j^k  \geq 0, \\
\nu_{pj}^{\text{-sum}}u_p^k u_j^k, &
\! \! \! \! \text{if }  \lambda_p \lambda_j \geq 0 \; \&\; u_p^k u_j^k  \leq 0, \\
\nu_{pj}^{\text{max}} |u_p^k u_j^k | , & \! \! \! \! \text{if }  \lambda_p \lambda_j < 0,
\end{cases}\label{apjdefn}
\end{equation}
where $\nu_{pj}^{\text{max}} = \max{\{ |\nu_{pj}^{\text{-sum}}|,  \nu_{pj}^{\text{+sum}}\}} $.

We define the {\em graph explore-exploit index} $\epsilon_n$ as
\begin{equation}
\epsilon_n =  \sqrt{M}  \sum_{p=2}^M \frac{|\lambda_p|}{1-|\lambda_p|},  \label{epsilonndef}
\end{equation}
and the {\em nodal explore-exploit centrality index} $\epsilon_c^k$ for node $k$ as 
\begin{equation}\label{eq:epsilon-c}
\epsilon_c^k   =    M \sum_{p=1}^M \sum_{j=2}^M \frac{|\lambda_p \lambda_j| }{1-|\lambda_p \lambda_j|} a_{pj}(k).
\end{equation}
Since $|\lambda_p|< 1$, for all $p \geq 2$, definitions \eqref{epsilonndef}-\eqref{eq:epsilon-c} imply that $\epsilon_n$ and $\epsilon_c^k$ decrease with a decrease in $|\lambda_p|$ for any $p \geq 2$. A small value of $\epsilon_n$ reflects a high level of symmetry and connectivity in the graph. This will be shown to predict low error in each agent's estimate of the average number of times a suboptimal arm has been chosen and thus high group explore-exploit performance. Dependence on the $k$-th component of the eigenvectors in \eqref{eq:epsilon-c} makes $\epsilon_c^k$ an index for node $k$, which measures how well agent $k$ estimates the second-order moments of rewards.  In an asymmetric graph, $\epsilon_c^k < \epsilon_c^l$ reflects a more favorable location in the graph for node $k$ as compared to node $l$. 

Both $\epsilon_n$ and $\epsilon_c^k$ depend only on the topology of the communication graph, yet they predict distributed cooperative estimation performance, as we show next, and  explore-exploit performance, as we show in subsequent sections.

\begin{proposition}[\bit{Performance of cooperative estimation}]\label{prop:coop-est}
For the distributed estimation algorithm defined in~\eqref{nhatdefn},~\eqref{shatdefn} and~\eqref{eqnmean}, and a doubly stochastic matrix $P$ defined in~\eqref{Pdefn}, the following statements hold:
\begin{enumerate}
\item the estimate $\hat n_i^k(t)$ satisfies
\begin{align*}
n_i^{\text{cent}}(t) - \epsilon_n \le  \hat{n}_i^k(t) \le   n_i^{\text{cent}}(t) + \epsilon_n;
\end{align*}
\item the following inequality holds for the estimate   $\hat {n}_i^k(t)$ and the sequence $\seqdef{\xi_i^j(\tau)}{\tau\in \until{t}}$, $ j \in \until{M}$:
\begin{equation*}
\sum_{\tau=1}^{t} \sum_{j=1}^M \left(\sum_{p=1}^M \lambda_p^{t-\tau+1}  u_p^k u_p^j  \right)^2 \xi_i^j(\tau) \leq \frac{ \hat n_i^k(t) + \epsilon_c^k}{M}.
\end{equation*}
\end{enumerate}
\end{proposition}
\begin{proof}
The proof uses some algebraic manipulations on the modal decomposition of~\eqref{nhatdefn}.  See~\ref{app:proof-prop1}. 
\end{proof}

We now  derive concentration bounds for the estimated mean computed with the cooperative estimation algorithm. Standard concentration inequalities, such as the Chernoff-Hoeffding inequality, rely on the sample size being independent of the realized values of samples. In the context of MABs, the arm selected at time $t$ depends on the rewards accrued at previous times. This makes the number of times an arm is sampled and the total reward accrued at that arm dependent random variables. For the case of a single agent, the specific kind of dependence between these random variables that occurs in MAB problems is leveraged to derive  a concentration inequality  in~\cite{AG-EM:08}.  In the following, we extend this concentration inequality to the distributed estimation algorithm studied here. 

For $i\in \{1,\dots,N\}$ and $k\in \{1,\dots,M\}$, $\seqdef{r_i^k(t)}{t \in \naturals}$  is a sequence of i.i.d.~sub-Gaussian {rewards} with mean $m_i \in \real$.   Let $\mathcal{F}_t$ be the filtration defined by the sigma-algebra of all the measurements until time $t$. Let $\seqdef{\xi_i^k(t)}{t \in \naturals}$  be a sequence of Bernoulli variables such that $\xi_i^k(t)$ is deterministically known given $\mc F_{t-1}$, i.e., $\xi_i^k(t)$ is pre-visible with respect to $\mc F_{t-1}$.
Let $\phi_i(\beta) = \lnn{\mathbb{E}\left[\expp{\beta r_i^k(t)}\right]}$ denote the cumulant generating function of $r_i^k(t)$.  

\begin{theorem}[\bit{Concentration bounds for the mean estimator}] \label{thm:EstDevBoudnsCondensed}
	For the estimates	$\hat{s}_i^k(t)$ and $\hat{n}_i^k(t)$ obtained using~\eqref{nhatdefn} and \eqref{shatdefn} given rewards drawn from a sub-Gaussian distribution as defined in Definition \ref{defn:sub-gaussian}, the following concentration inequality holds: 
	\begin{equation}
		\mathbb{P} \left( \frac{\hat{s}_i^k(t) \!-\! m_i \hat{n}_i^k(t)}{\left(\frac{1}{M} \left(\hat{n}_i^k(t) + \epsilon_c^k\right)\right)^{\slfrac{1}{2}}} > \delta  \right) <\Bigg\lceil \frac{\lnn{t \!+\! \epsilon_n}}{\lnn{1 + \eta}} \Bigg\rceil  \expp{\frac{-\delta^2}{2 \sigma_g^2} G(\eta) } ,
	\end{equation}
	where 	$\delta > 0$, $\eta \in (0,4)$, $G(\eta) = (1-\frac{\eta^2}{16})$, and $\epsilon_n$ and $\epsilon_c^k$  are defined in~\eqref{epsilonndef} and~\eqref{eq:epsilon-c}, respectively. 
\end{theorem}
\begin{proof}
The proof recursively computes a moment generating function of $\hat s_i^k(t)$ using modal decomposition of~\eqref{shatdefn} and conditioning on the appropriate filtration. It subsequently uses the Markov inequality and a peeling argument based on the union bound to establish the  inequality. See~\ref{app:proof-thm-1}.
\end{proof}

\section{Cooperative Decision Making:   Unconstrained Reward 
}\label{sec:DistributedDecisionMaking}
{In this section, we extend the UCB algorithm~\cite{PA-NCB-PF:02}, for single-agent decision making among arms, to design decision making in the distributed cooperative setting in which a group of $M$ agents communicate with one another over a network with fixed graph. At every time $t$, each agent $k$ updates its estimates of the mean rewards at each arm $i$ according to the cooperative estimation algorithm of Section~\ref{sec:coop-est}. Then each agent chooses an arm to maximize its own individual reward.  We consider the case of unconstrained sub-Gaussian rewards here and the case of constrained sub-Gaussian rewards in Section~\ref{sec:Coop-ucb2-collisions}.  

Intuitively, each agent will perform better with communication than without.  However, the extent of the performance advantage of each agent and  the group as whole, as a result of communication, depends on the network structure. We compute bounds on group performance by computing bounds on the expected group cumulative regret, and we show how the bounds depend on graph explore-exploit index $\epsilon_n$ and nodal explore-exploit centrality indices $\epsilon_c^k$, $k = 1, \ldots, M$.}

\subsection{The coop-UCB2 Algorithm}
The coop-UCB2 algorithm is initialized by each agent sampling each arm once and proceeds as follows (see \ref{app:pseudocode} for pseudocode implementation). At time $t$ each agent $k$ selects the arm with maximum $Q_i^{k}(t-1) = \hat{\mu}_i^{k}(t-1) + C_i^k(t-1)$, where
\begin{equation}
C_i^k(t-1) =  \sigma_g \; \sqrt[]{\frac{2\gamma}{G(\eta)} \cdot \frac{\hat{n}_i^{k}(t-1) +  f(t-1)}{M\hat{n}_i^{k}(t-1)}\cdot\frac{ \lnn{t-1}}{\hat{n}_i^{k}(t-1)}} \label{C2defn-bounded}
\end{equation}
for sub-Gaussian rewards.  Here, $f(t)$ is an increasing sublogarthmic function of $t$, $\gamma>1$, $\eta \in (0,4)$, and $G(\eta) = 1-\slfrac{\eta^2}{16}$. 

Then, at each time $t$, each agent $k$ updates its cooperative estimate of the mean reward at each arm using the distributed cooperative estimation algorithm described in~(\ref{shatdefn})-(\ref{eqnmean}). Note that the heuristic $Q_i^k$ requires agent $k$ to know the total number of agents $M$ but {nothing about the} 
graph structure.


\begin{theorem}[\bit{Upper Bound on Suboptimal Selections for coop-UCB2 Algorithm}]\label{thm:regret-coop-ucb2}
For the coop-UCB2 algorithm and the distributed cooperative multi-agent MAB problem under the unconstrained reward model with sub-Gaussian rewards, the number of times a suboptimal arm $i$ is selected by all agents until time $T$ satisfies
\[
	\sum_{k=1}^M\mathbb{E}[n_i^{k}(T)]  \leq \frac{4 \sigma_g^2 \gamma \ln T}{\Delta_i^2 G(\eta)} \left( 1 + \sqrt{1 + \frac{\Delta_i^2 M G(\eta)}{2\gamma \sigma_s^2} \frac{f(T)}{\ln T}} \right) +L, 
\]
where 
\begin{align}
& L(\epsilon_n, \epsilon_c^1, \ldots, \epsilon_c^M) =  \sum_{k=1}^M  (t_k^\dagger -1) + M(1+\epsilon_n) +1  \nonumber \\
& \quad + \frac{2M}{\lnn{1 +\eta}}  \bigg( \frac{1}{(\gamma - 1)^2}  + {\frac{\gamma \lnn{1 + \epsilon_n)(1 + \eta)}}{\gamma - 1} +1}\bigg),
\label{eq:L}
\end{align}
is a constant independent of $T$ and $t^\dagger_k= f^{-1}(\epsilon_c^k)$. 
\end{theorem}

\begin{proof}
The upper bound is computed 
as for UCB1~\cite{PA-NCB-PF:02} and leverages Proposition~\ref{prop:coop-est} and Theorem~\ref{thm:EstDevBoudnsCondensed}.  
See~\ref{app:proof-thm2}.  
\end{proof}

\begin{corollary}[\bit{Regret of the coop-UCB2 Algorithm}]\label{corr:regret-coop-ucb2}
For the coop-UCB2 algorithm and the distributed cooperative multi-agent MAB problem under the unconstrained reward model with sub-Gaussian rewards, the expected cumulative group regret until time $T$ satisfies
\[
\supscr{R}{unc}(T)	\leq \sum_{i=1}^N \frac{4 \sigma_g^2 \gamma \ln T}{\Delta_i G(\eta)} \left( 1 + \sqrt{1 + \frac{\Delta_i^2 M G(\eta)}{2\gamma \sigma_s^2} \frac{f(T)}{\ln T}} \right) + \sum_{i=1}^N L \Delta_i.
\]
\end{corollary}
\begin{proof}
    The corollary follows by substituting the upper bound on $\sum_{k=1}^M \expt[n_i^k(T)]$ from Theorem~\ref{thm:regret-coop-ucb2} into  \eqref{Runc}. 
\end{proof}

From these bounds, we can compare performance for the distributed case relative to the centralized case,  
and we can draw conclusions about the predictive value of explore-exploit indices $\epsilon_n$ and $\epsilon_c^k$ as follows. 

\begin{remark}[\bit{Asymptotic Regret for coop-UCB2}]
	In the limit $t\rightarrow + \infty$, $\frac{f(t)}{\ln(t)} \to 0^+$, $\eta \rightarrow 0$, and   
	\begin{equation*}
	\sum_{k=1}^M \mathbb{E}[n_i^k(T)] \leq  \left( \frac{8 \sigma_g^2  \gamma}{\Delta_i^2} + o(1) \right)\ln T.   
	\end{equation*}
	We thus recover the upper bound on regret for a {centralized fusion center} 
	as given in \eqref{eqn:fusioncenterregret} within a constant factor.
	\oprocend
\end{remark}

\begin{remark}[\bit{{Predicting Relative Performance from Network Graph Topology}}] \label{Remark:Indiv}
	Theorem \ref{thm:regret-coop-ucb2} and Corollary~\ref{corr:regret-coop-ucb2} provide bounds on the performance of the group as a function of the graph structure, {as measured by the group explore-exploit index $\epsilon_n$ and nodal explore-exploit centrality indices $\epsilon_c^k$. While the logarithmic term in the upper bound on group performance is independent of graph structure, the sublogarithmic term $L$, given in \eqref{eq:L}, depends on $\epsilon_n$ and $\epsilon_c^k$. Our theory predicts that the performance of a group is better for a network with smaller $\epsilon_n$, since a smaller $\epsilon_n$ implies a smaller upper bound on expected cumulative group regret.  Likewise, our theory predicts that the performance of individual agent $j$ is better than the performance of individual agent $l$ if $\epsilon_c^j < \epsilon_c^l$, since a smaller $\epsilon_c^k$ implies a smaller contribution from agent $k$ to the upper bound on expected cumulative group regret. These predictions rely on the bounds being sufficiently tight; we illustrate the usefulness of the predictions with simulations in Section~\ref{NetworkPerformanceAnalysis}}.
	\oprocend
\end{remark}

\section{Cooperative Decision Making: Constrained Reward
}\label{sec:Coop-ucb2-collisions}
In this section we extend our analyses in Section \ref{sec:DistributedDecisionMaking} to the case of the constrained reward model\footnote{Some authors  ~\cite{anandkumar2011distributed,kalathil2014decentralized} have considered the case where agents that sample the same arm at the same time receive a split reward.  The algorithm presented here is still appropriate for that scenario, and the regret as defined above will upper bound the regret in the case of split rewards.}.  In this setting the optimal solution in terms of group regret is for the $M$ agents to each sample a different arm from among the $M$-best arms at every time $t$.  Recall that $\mathcal{O}_{k}^*$ is the set of $k$-best arms. Let $\Delta_{\min} = \min \setdef{|m_i-m_j|}{i, j \in \until{N}, i \ne j}$. In the following, we assume that each agent $k$  has a preassigned unique rank $\omega^k \in \until{M}$ and will attempt to sample the arm with the $\omega^k$-th best reward.  Without loss of generality, we assume that $\omega^k =k$. We  define agent $k^i$ {as the index of} the agent attempting to sample arm $i \in \mc O^*_M$.  We let $k^i = 0$ if $i \notin \mc O^*_M$. 
Therefore, the expected cumulative regret of agent $k$ at time $T$ is
\begin{equation} \label{eqn:RegretDefnCol}
	R^k(T) = \sum_{t=1}^T \left [  m_{b^k}- \mathbb{E} \left [ \sum_{i=1}^N r_i^k(t) \mathds{1}\{i^k(t) = i\} \mathbb{I}_{i}^k(t) \right ] \right ],
\end{equation}
where $\mathbb{I}_{i}^k(t) = 1$ if agent $k$ is the only agent to sample arm $i$ at time $t$, and $0$ otherwise.  

In the following, we assume that while agents do not receive any reward if they sample the same arm, they still have access to the {value of the} reward they did not receive and they can use it in updating their estimates of the mean rewards.

\subsection{The coop-UCB2-selective-learning Algorithm}
In this section, we present the coop-UCB2-selective-learning algorithm in which agent $k$ selectively targets the $k$-th best arm (see \ref{app:pseudocode2} for pseudocode implementation).  The coop-UCB2-selective-learning algorithm for sub-Gaussian rewards is initialized by each agent sampling each arm once in a round-robin fashion with agent $k$ begining the sampling with the $k$-th arm. At each time $t$, each agent $k$ updates its cooperative estimate of the mean reward at each arm using the distributed cooperative estimation algorithm described in~\eqref{shatdefn}--\eqref{eqnmean}. 

 Subsequently, at time $t$, each agent $k$ estimates $\mc O_k^*$ by constructing the set $\mathcal{O}_{k}(t)$ containing $k$ arms associated with the indices of the $k$ highest values in the set $\setdef{Q_i^{k}(t-1) = \hat{\mu}_i^{k}(t-1) + C_i^k(t-1)}{i \in \until{N}}$, where
\begin{equation}
C_i^k(t-1) =  \sigma_g \;\sqrt[]{\frac{2\gamma}{G(\eta)} \cdot \frac{\hat{n}_i^{k}(t-1) +  f(t-1)}{M\hat{n}_i^{k}(t-1)}\cdot\frac{ \lnn{t-1}}{\hat{n}_i^{k}(t-1)}}, \label{C2defn-bounded-col}
\end{equation}
 $f(t)$ is an increasing sublogarthmic function of $t$, $\gamma>1$, $\eta \in (0,4)$, and $G(\eta) = 1-\slfrac{\eta^2}{16}$. 

Each agent $k$ then selects the arm associated  with the minimum value in the set $\setdef{W_i^{k}(t-1) = \hat{\mu}_i^{k}(t-1) - C_i^k(t-1)}{i \in \mc O_k}$. 

Our algorithm generalizes the selective-learning algorithm for multi-agent MABs with no communication among agents proposed in~\cite{gai2014distributed} to the case of communicating agents.

\subsection{Analysis of the coop-UCB2-selective-learning Algorithm}
We first bound the number of times an arm $i$ is incorrectly selected. We call the selection of arm $i \in \mc O_M^*$ incorrect if it is selected by an agent $k \ne k^i$. Any selection of arm  $i \notin \mc O_M^*$ is incorrect. Let $\bar n_i^k(t)$ be the number of incorrect selections of arm $i$ until time $t$. 

\begin{theorem}[\bit{Upper Bound on Incorrect Selections for coop-UCB2-selective-learning Algorithm}]\label{thm:incorrect-selections}
For the coop-UCB2-selective-learning algorithm and the distributed cooperative multi-agent MAB problem under the constrained reward model with sub-Gaussian rewards, the number of times an arm $i$ is incorrectly selected by all agent until time $T$ satisfies
		\begin{equation*}
	\sum_{k \neq k^i}  \mathbb{E} \left[ \bar  n_i^k(T)  \right] \leq   \frac{4 \sigma_g^2 \gamma  }{\Delta_{\min}^2 G(\eta)}  \left( 1 + \sqrt{1  +  \frac{\Delta_{\min}^2 MG(\eta)}{2 \sigma_g^2 \gamma} \frac{f(T)}{\ln T}} \right)  \ln T + \bar L,
	\end{equation*}
	where
	\begin{align}
	    & \bar L(\epsilon_n, \epsilon_c^1, \ldots, \epsilon_c^M) = \sum_{k=1}^M (t_k^\dag -1)+  M (1+\epsilon_n) + 1
	     \nonumber \\
	    & \quad + \frac{2M(N+1)}{\lnn{1\!+\!\eta}}   \left( \frac{1}{(\gamma - 1)^2} +  {\frac{\gamma \lnn{1 + \epsilon_n)(1 + \eta)}}{\gamma - 1} +1} \right) 
	    \label{eq:Lbar}
	\end{align}
is a constant independent of $T$ and $t^\dagger_k= f^{-1}(\epsilon_c^k)$. 	
%
%
\end{theorem}
\begin{proof}
The upper bound is computed similarly to $\mathrm{SL}(K)$~\cite{gai2014distributed}, leveraging Proposition~\ref{prop:coop-est} and Theorem~\ref{thm:EstDevBoudnsCondensed}.  See~\ref{app:proof-thm3}.  
\end{proof}

\begin{corollary}[\bit{Regret of the coop-UCB2-selective-learning Algorithm}] \label{thm:RegretCollisions}
For the coop-UCB2-selective-learning algorithm and the distributed cooperative multi-agent MAB problem under the constrained reward model with sub-Gaussian rewards, the expected cumulative regret of the group satisfies
	\begin{equation}
	\supscr{R}{con}(T) \le 	\sum_{k=1}^M \! R^k(T)  \leq  m_{i^*} N B + \sum_{k=1}^M  m_{b^k}  B \nonumber
	\end{equation}
		where 
	\begin{align}
		&B  \! =\! \frac{4 \sigma_g^2 \gamma  }{\Delta_{\min}^2 G(\eta)}  \left( 1 + \sqrt{1  +  \frac{\Delta_{\min}^2 MG(\eta)}{2 \sigma_g^2 \gamma} \frac{f(T)}{\ln T}} \right)  \ln T + \bar L.
	\end{align}
\end{corollary}

\begin{proof}
	As in \cite{gai2014distributed}, agent $k$ incurs regret either by selecting an arm $i \neq b^k$ or when another user $j\neq k$ selects arm $b^k$.  Therefore, 
	\begin{align}
		\sum_{k=1}^M \! R^k(T)  & 
		 \leq \! \sum_{k=1}^M \sum_{i \neq b^k} \! \mathbb{E} \left [ \bar n_i^k(T) \right ] \! m_{b^k} + \sum_{k=1}^M \sum_{j \neq k} \mathbb{E} \! \left [ \bar n_{b^k}^j(T) \right ] \! m_{b^k} \nonumber \\
		& \leq \! m_{i^*} \sum_{i=1}^N \sum_{k \neq k^i} \! \mathbb{E} \left [ \bar n_i^k(T) \right ]  + \sum_{k=1}^M \sum_{j \neq k} \mathbb{E} \! \left [ \bar n_{b^k}^j(T) \right ] \! m_{b^k} \nonumber \\
		& \leq \! m_{i^*} \sum_{i=1}^N B + \sum_{k=1}^M  m_{b^k}  B, \nonumber
	\end{align}
	completing the proof.
\end{proof}

From these bounds, we can compare performance in the case of communication between agents relative to the case of no communication between agents, and we can draw conclusions about the predictive value of explore-exploit indices $\epsilon_n$ and $\epsilon_c^k$ for the unconstrained reward model, as follows.

\begin{remark}[\bit{Concise Upper Bound on Regret}] \label{Remark:RegretCollisionsConcise}
	The upper bound on expected cumulative group regret in Corollary \ref{thm:RegretCollisions} can be expressed concisely, at the expense of some tightness, as
	\begin{equation}
		\sum_{k=1}^M \!  R^k(T)   \leq  m_{i^*} B (M+N). \nonumber
	\end{equation}
	In the limit $\eta \to 0^+$ and $\gamma \to 1^+$, this is a factor of $4M$ tighter than the bounds in \cite{gai2014distributed}, demonstrating the benefits of communication between agents for the constrained reward model.
\end{remark}

\begin{remark}[\bit{{Predicting Relative Performance from Network Graph Topology for Constrained Reward Model}}] \label{Remark:Indiv_Constrained}
	Theorem~\ref{thm:incorrect-selections} and Corollary~\ref{thm:RegretCollisions} predict the performance of the group as a function of the graph structure for the constrained reward model just as described for the unconstrained reward model in Remark~\ref{Remark:Indiv}, since $\bar L$ given in \eqref{eq:Lbar} has the same form as $L$ given in \eqref{eq:L}.
	\oprocend
\end{remark}

\definecolor{Hue1}{rgb}{0.89, 0.1, 0.11}
\definecolor{Hue2}{rgb}{0.22, 0.55, 0.72}
\definecolor{Hue3}{rgb}{0.30, 0.69, 0.29}
\definecolor{Hue4}{rgb}{0.64, 0.31, 0.64}
\definecolor{Hue5}{rgb}{1, 0.50, 0}

\makeatletter
\tikzset{circle split part fill/.style args={#1,#2}{%
		alias=tmp@name, 
		postaction={%
			insert path={
				\pgfextra{%
					\pgfpointdiff{\pgfpointanchor{\pgf@node@name}{center}}%
					{\pgfpointanchor{\pgf@node@name}{east}}%
					\pgfmathsetmacro\insiderad{\pgf@x}
					\fill[#1] (\pgf@node@name.base) ([xshift=-\pgflinewidth]\pgf@node@name.east) arc
					(0:180:\insiderad-\pgflinewidth)--cycle;
					\fill[#2] (\pgf@node@name.base) ([xshift=\pgflinewidth]\pgf@node@name.west)  arc
					(180:360:\insiderad-\pgflinewidth)--cycle;            
				}}}}}
				\makeatother
				
				\def \minsize = {2cm}
				\tikzstyle{SimpleNetwork}=[draw,circle,minimum width=6pt]
				\tikzstyle{Leader_CALI} = [draw, diamond, fill = green, minimum size=0.55cm]
				\tikzstyle{Leader_FYL}  = [draw, diamond, fill = blue, minimum size=0.55cm]
				\tikzstyle{Leader_Both} = [draw, circle split, circle split part fill={green,blue}, minimum size=0.5cm]
				\tikzstyle{Follower} = [draw, circle, fill = white, minimum size=0.5cm]
				
				\tikzstyle{F_DandF_I^w} = [draw, circle split, circle split part fill={magenta,blue}, minimum size=0.5cm]
				\tikzstyle{F_DandF_D^w} = [draw, circle split, circle split part fill={magenta,cyan}, minimum size=0.5cm]
				\tikzstyle{F_IandF_I^w} = [draw, circle split, circle split part fill={white,blue}, minimum size=0.5cm]
				\tikzstyle{F_D} = [draw, regular polygon, regular polygon sides=3, fill = magenta, minimum size=0.5cm]
				\tikzstyle{F_D^w} = [draw, circle, fill = cyan, minimum size=0.5cm]
				\tikzstyle{F_I^w} = [draw, circle, fill = blue, minimum size=0.5cm]
				\tikzstyle{F_I} = [draw, circle, fill = cyan, minimum size=0.5cm]
				\tikzstyle{Generic} = [draw, circle, fill = white, minimum size=0.5cm]
				
				\tikzstyle{Leader} = [draw, diamond, fill = green, minimum size=0.55cm]
				\tikzstyle{Fdirect} = [draw, regular polygon, regular polygon sides=3, fill = magenta, minimum size=0.5cm]
				\tikzstyle{F} = [draw, circle, fill = cyan, minimum size=0.5cm]
				
				\tikzstyle{TestNode} = [draw, isosceles triangle, isosceles triangle apex angle=90, shape border rotate=90, fill = green, minimum size=0.25cm]
				\tikzstyle{TestNode2} = [draw, isosceles triangle, isosceles triangle apex angle=90, shape border rotate=-90, fill = blue, minimum size=0.25cm]
				
				\tikzstyle{Generic1} = [draw, circle, fill = Hue1, minimum size=0.5cm]
				\tikzstyle{Generic2} = [draw, circle, fill = Hue2, minimum size=0.5cm]
				\tikzstyle{Generic3} = [draw, circle, fill = Hue3, minimum size=0.5cm]
				\tikzstyle{Generic4} = [draw, circle, fill = Hue4, minimum size=0.5cm]
				\tikzstyle{Generic5} = [draw, circle, fill = Hue5, minimum size=0.5cm]

\section{Numerical Illustrations}
\label{NetworkPerformanceAnalysis}
\def \GraphWidth {.75\linewidth}

In this section, we illustrate our theoretical analyses from the previous sections 
with numerical examples. We first  provide examples in which the ordering of the performance of nodes obtained through numerical simulations is as predicted by the ordering of the nodal explore-exploit centrality indices, as discussed in Remarks~\ref{Remark:Indiv} and \ref{Remark:Indiv_Constrained}. That is, a smaller $\epsilon_c^k$ predicts better performance for agent $k$.  We then provide examples in which the ordering over networks of the performance of a group of agents is as predicted by the ordering over networks of the graph explore-exploit index, as discussed in Remark~\ref{Remark:Indiv} and \ref{Remark:Indiv_Constrained}. That is, a smaller $\epsilon_n$ predicts better performance for the group with the corresponding network graph. Our final example illustrates how performance improves with connectivity.

Unless otherwise noted in the simulations, we consider a $10$-arm bandit problem with mean rewards drawn from a normal random distribution for each Monte-Carlo run with mean $0$ and standard deviation $10$. The sampling standard deviation  is $\sigma_s = 30$ and the results displayed are the average of $10^6$ Monte-Carlo runs.  These parameters were selected to give illustrative results within the displayed time horizon, but the relevant conclusions hold across a wide range of parameter values. In the simulations $f(t)=\sqrt{\ln t}$, and consensus matrix $P$ is as in \eqref{Pdefn} with $\kappa = \frac{d_{\text{max}}}{d_{\text{max}}-1}$.


\begin{example}\label{ex:4agent_nodecompare}
	Figure \ref{fig:4agent_nodecompare} demonstrates the ordering of performance among agents using coop-UCB2 with the underlying graph structure in Table \ref{table:4agentgraph}.  The values of $\epsilon_c^k$ for each node are also given in Table \ref{table:4agentgraph}. As predicted by Theorem \ref{thm:regret-coop-ucb2} (Remark~\ref{Remark:Indiv}), agent $1$ should have the lowest regret, agents $2$ and $3$ should have equal and intermediate regret, and agent $4$ should have the highest regret as this is their ordering with respect to $\epsilon_c^k$. These predictions are validated in our simulations shown in Figure \ref{fig:4agent_nodecompare}.
\end{example}

\begin{figure}[ht!]
	\centering
	\begin{subfigure}{.5\linewidth}
		\centering
		\resizebox{0.9\linewidth}{!}{
			\begin{tikzpicture}[ every node /.style=minimum size=4em]
			
			\def \n {5}
			\def \offset = {30}
			\def \radius {1cm}
			\def \margin {8} 
			\tikzstyle{every node}=[font=\tiny]
			
			\node[Generic1] (1) at (0:0) {1};
			\node[Generic2] (2) at (270:\radius) {2};
			\node[Generic3] (4) at (330:\radius) {3};
			\node[Generic4] (3) [left of = 1]  {4};
			
			\draw (1) -- (2);
			\draw (1) -- (3);
			\draw (1) -- (4);
			\draw (2) -- (4);
			\end{tikzpicture}}
	\end{subfigure}%
	\begin{subfigure}{.5\linewidth}
		\begin{tabular}{c | c  | c}
			Agent & Degree  & $\epsilon_c^k$ \\ \hline
			1     & 3                & 0  \\ 
			2     & 2               & 2.31  \\ 
			3     & 2               & 2.31  \\ 
			4     & 1               & 5.41  
		\end{tabular}
	\end{subfigure}
	\captionof{table}{Fixed network used in Example \ref{ex:4agent_nodecompare}.}
	\label{table:4agentgraph}
\end{figure}



	


\addtocounter{figure}{-1}
\begin{figure}[ht!]
	\centering
	\includegraphics[width=\GraphWidth]{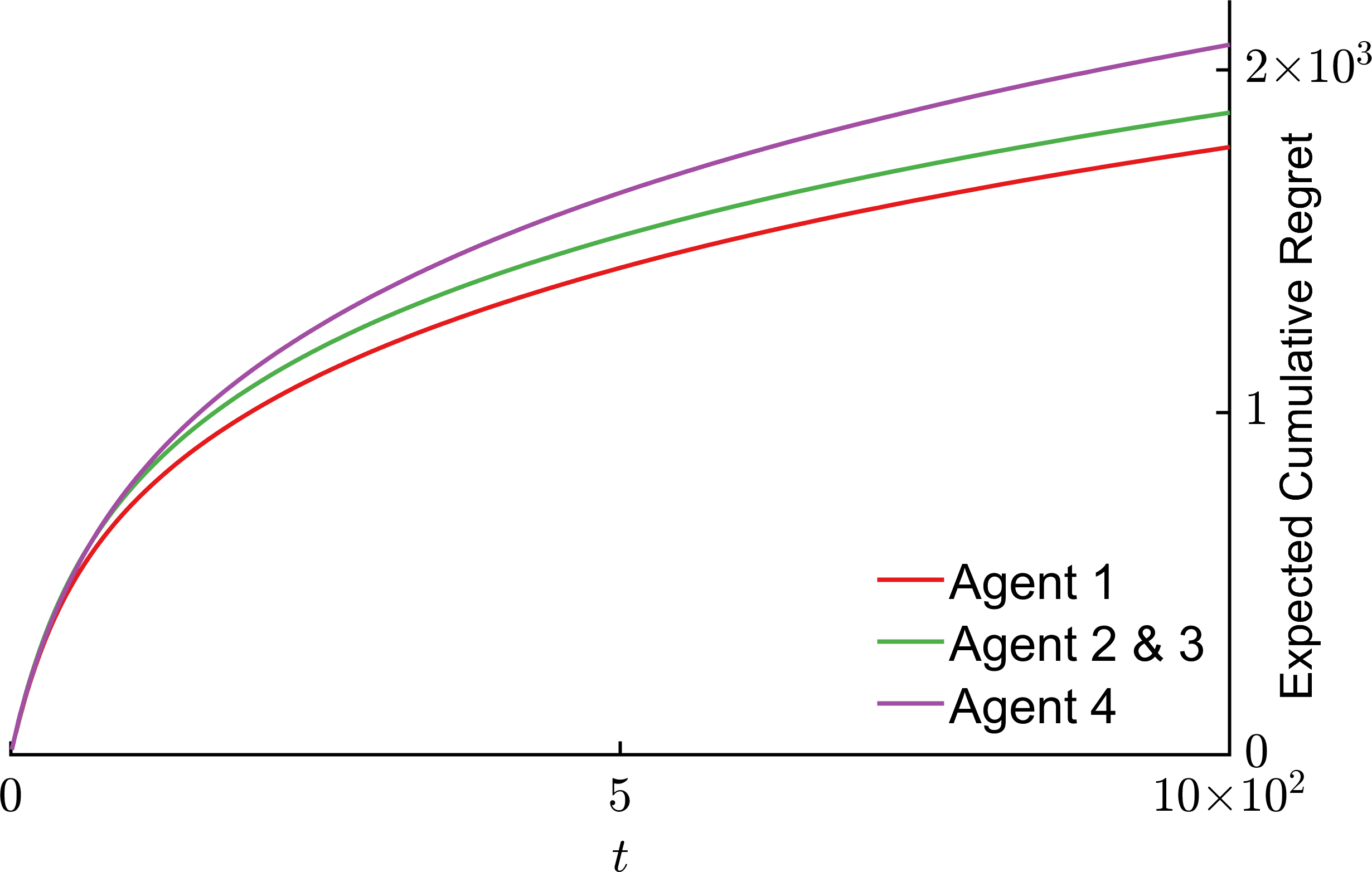}
	\caption[Simulation results comparing expected cumulative regret for agents in the fixed graph shown in Table \ref{table:4agentgraph}.]{Simulation results comparing expected cumulative regret for agents in the fixed network shown in Table \ref{table:4agentgraph}.  Agents $2$ and $3$, with the same centrality index, have nearly identical expected regret. {Agent 1, with lowest centrality index, performs best and agent 4, with highest centrality index, performs worst.}}
	\label{fig:4agent_nodecompare}
\end{figure}

\begin{example}\label{ex:house_nodecompare}
	Figure \ref{fig:house_nodecompare} demonstrates the ordering of performance among agents using coop-UCB2 with the underlying graph structure in Table \ref{table:5agentgraph}. Rewards are drawn from a normal distribution with mean $0$ and standard deviation $5$.  The values of $\epsilon_c^k$ for each node are also given in Table \ref{table:5agentgraph}, along with the values of degree and information centrality for each node \cite{Poulakakis2015}, for comparison.  Degree centrality for a node is defined as the number of neighbors.  Information centrality, defined in Stephenson and Zelen \cite{STEPHENSON19891}, is a nodal measure of the ``effective resistance" between the  node and  every other node in the network.

\begin{figure}[ht!]
	\centering
	\begin{subfigure}{.35\linewidth}
		\centering
		\resizebox{0.9\linewidth}{!}{
			\begin{tikzpicture}[ every node /.style=minimum size=4em]
			
			\def \n {5}
			\def \offset = {30}
			\def \radius {1cm}
			\def \margin {8} 
			\tikzstyle{every node}=[font=\tiny]
			
			\node[Generic1] (1) at (0:0) {1};
			\node[Generic2] (2) at (270:\radius) {2};
			\node[Generic5] (5) at (330:\radius) {5};
			\node[Generic3] (3) [left of = 1]  {3};
			\node[Generic4] (4) [below of = 3]  {4};
			
			\draw (1) -- (2);
			\draw (1) -- (3);
			\draw (1) -- (5);
			\draw (2) -- (5);
			\draw (2) -- (4);
			\draw (4) -- (3);
			\end{tikzpicture}}
	\end{subfigure}%
	\begin{subfigure}{.65\linewidth}
		\begin{tabular}{c | c | c | c}
			Agent & Degree & Info. Cent. & $\epsilon_c^k$ \\ \hline
			1     & 3      & .35           & 1.4  \\ 
			2     & 3      & .35          & 1.4  \\ 
			3     & 2      & .28          & 3.4  \\ 
			4     & 2      & .28           & 3.4  \\
			5     & 2      & .27           & 2.9 
		\end{tabular}
	\end{subfigure}
	\captionof{table}{Fixed network used in Example \ref{ex:house_nodecompare} and several centrality indices.}
	\label{table:5agentgraph}
\end{figure}

\addtocounter{figure}{-1}
\begin{figure}[ht!]
	\centering
	\includegraphics[width=\GraphWidth]{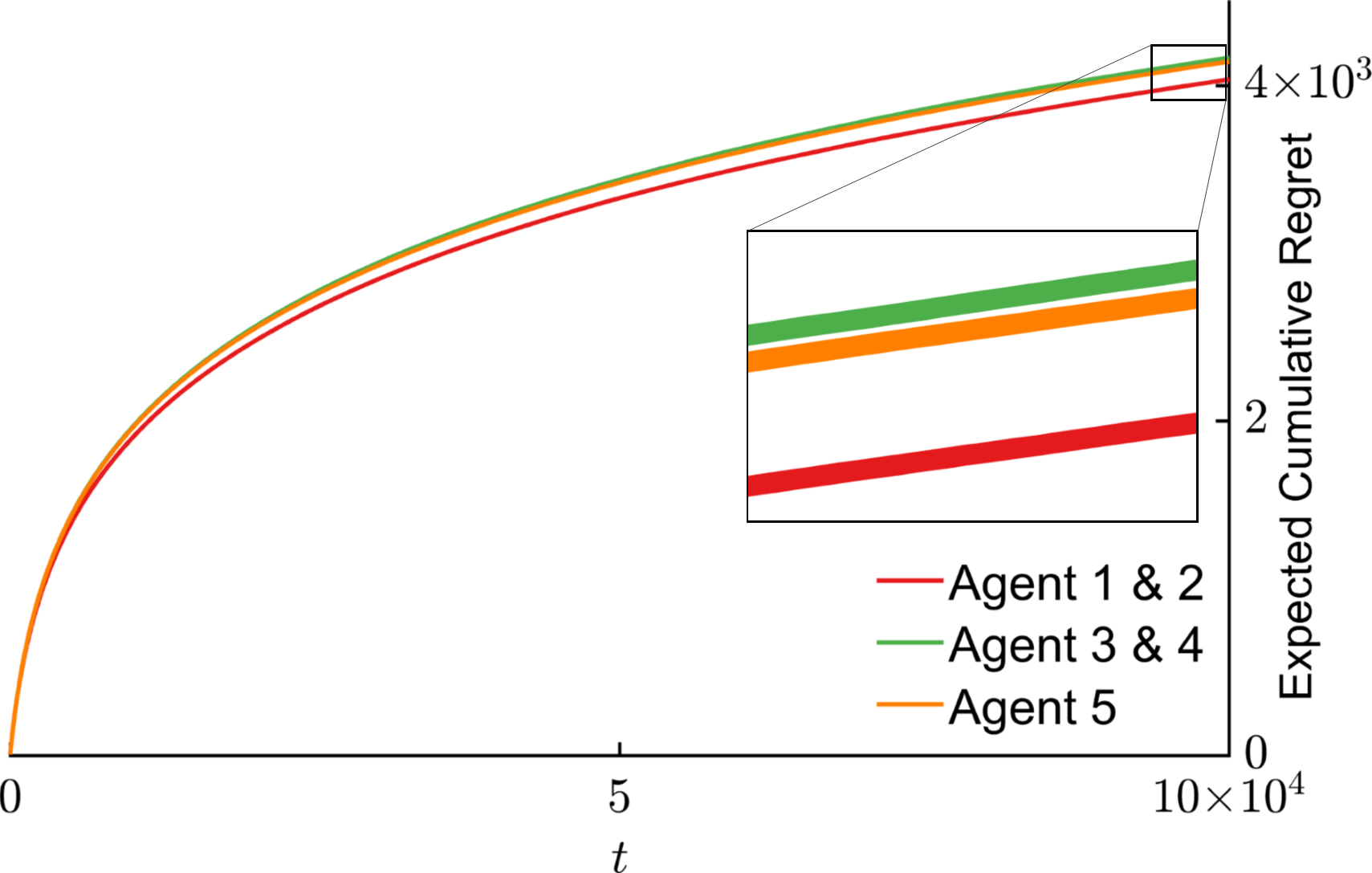}
	\caption{Simulation results comparing expected cumulative regret for agents in the fixed graph  shown in Table \ref{table:5agentgraph}.  }
	\label{fig:house_nodecompare}
\end{figure}

	For this example, degree centrality does not distinguish agent $5$ from agents $3$ and $4$, whereas  $\epsilon_c^k$ (and information centrality) does.  Further, according to information centrality, which is larger the more central the node, node $5$ is less information central than nodes $3$ and $4$.  In contrast, according to  $\epsilon_c^k$, which is smaller the more central the node, node $5$ is more explore-exploit central than nodes $3$ and $4$. 
	
	{As in the prior example, the simulation results of Figure~\ref{fig:house_nodecompare} validate the prediction of Theorem \ref{thm:regret-coop-ucb2} (Remark~\ref{Remark:Indiv}) that the ordering of agents by performance, as measured by expected cumulative regret, is the same as the ordering of agents by nodal explore-exploit centrality index $\epsilon_c^k$, with smaller $\epsilon_c^k$ corresponding to lower regret.  
	In contrast, for this example, the ordering of agents by degree or information centrality do not predict the ordering of agents by performance.}
	
    We have found some parameter regimes, specifically for rewards that are far apart in mean value, where information centrality does give the correct ordering of performance, rather than $\epsilon_c^k$.  This is likely due to sensitivity of performance to the $\Delta_i$.  However, we have observed that $\epsilon_c^k$ is broadly predictive of performance for a variety of regimes and network graphs. 
\end{example}

	


	

\subsection{{Validation of Relative Performance of Networks as Predicted by Graph Explore-Exploit Index $\epsilon_n$}}

\begin{figure}[ht!]
	\centering
	\def \smallplotwidth {0.3\linewidth}
	
	\def \n {5}
	\def \offset = {30}
	\def \radius {1cm}
	\def \margin {8} 

	\begin{subfigure}[b]{\smallplotwidth}
		\centering
		
		\begin{tikzpicture}[ every node /.style=minimum size=4em, inner sep=0pt]		
		
		\node[Generic1] (1) at (18:\radius) {\tiny $\star$};
		\node[Generic1] (2) at (90:\radius) {\tiny $\star$};
		\node[Generic1] (3) at (162:\radius) {\tiny $\star$};
		\node[Generic1] (4) at (234:\radius)  {\tiny $\star$};
		\node[Generic1] (5) at (306:\radius)  {\tiny $\star$};
		
		\draw (1) -- (2);
		\draw (1) -- (3);
		\draw (1) -- (4);
		\draw (1) -- (5);
		\draw (2) -- (3);
		\draw (2) -- (4);
		\draw (2) -- (5);
		\draw (3) -- (4);
		\draw (3) -- (5);
		\draw (4) -- (5);
		\end{tikzpicture}
		\caption*{All-to-All \\ $\epsilon_n \approx 439$}
	\end{subfigure}%
	\begin{subfigure}[b]{\smallplotwidth}
		\centering
		
		\begin{tikzpicture}[ every node /.style=minimum size=4em, inner sep=0pt]

		\node[Generic2] (1) at (18:\radius) {\tiny $\star$};
		\node[Generic2] (2) at (90:\radius) {\tiny $\star$};
		\node[Generic2] (3) at (162:\radius) {\tiny $\star$};
		\node[Generic2] (4) at (234:\radius)  {\tiny $\star$};
		\node[Generic2] (5) at (306:\radius)  {\tiny $\star$};
		
		\draw (1) -- (2);
		\draw (2) -- (3);
		\draw (3) -- (4);
		\draw (4) -- (5);
		\draw (5) -- (1);
		\end{tikzpicture}
		\caption*{Ring \\ $\epsilon_n \approx 663$}
	\end{subfigure}%
	\begin{subfigure}[b]{\smallplotwidth}
	\centering
	
	\begin{tikzpicture}[ every node /.style=minimum size=4em, inner sep=0pt]		
	
		\node[Generic3] (1) at (0:0) {\tiny $\star$};
		\node[Generic3] (2) [right of = 1] {\tiny $\star$};
		\node[Generic3] (3) [below of = 1] {};
		\node[Generic3] (4) [right of = 3]  {};
		\node[Generic3] (5) at (60:\radius)  {};
		
		\draw (1) -- (2);
		\draw (2) -- (4);
		\draw (3) -- (4);
		\draw (5) -- (1);
		\draw (2) -- (5);
		\draw (3) -- (1);
		\end{tikzpicture}
		\caption*{House \\ $\epsilon_n \approx 724$}
	\end{subfigure}%
	\\
	\begin{subfigure}[b]{\smallplotwidth}
		\centering
		
		\begin{tikzpicture}[ every node /.style=minimum size=4em, inner sep=0pt]		
		
		\node[Generic4] (1) at (18:\radius) {};
		\node[Generic4] (2) at (90:\radius) {\tiny $\star$};
		\node[Generic4] (3) at (162:\radius) {};
		\node[Generic4] (4) at (234:\radius)  {};
		\node[Generic4] (5) at (306:\radius)  {};
		
		\draw (1) -- (2);
		\draw (2) -- (3);
		\draw (3) -- (4);
		\draw (5) -- (1);
		\end{tikzpicture}
		\caption*{Line \\ $\epsilon_n \approx 1334$}
	\end{subfigure}%
	\begin{subfigure}[b]{\smallplotwidth}
	\centering
	
		\begin{tikzpicture}[ every node /.style=minimum size=4em, inner sep=0pt]		
		
		\node[Generic5] (1) at (0:0) {\tiny $\star$};
		\node[Generic5] (2) at (45:\radius) {};
		\node[Generic5] (3) at (135:\radius) {};
		\node[Generic5] (4) at (225:\radius)  {};
		\node[Generic5] (5) at (315:\radius)  {};
		
		\draw (1) -- (2);
		\draw (1) -- (3);
		\draw (1) -- (4);
		\draw (1) -- (5);
		\end{tikzpicture}
		\caption*{Star \\ $\epsilon_n \approx 1781$}
	\end{subfigure}%

	\captionof{table}[Fixed networks used in Examples \ref{ex:5node_graphcompare} and \ref{ex:5node_graphcompare_bestagent}.]{Fixed networks used in Example \ref{ex:5node_graphcompare} arranged in order of increasing value of $\epsilon_n$.  Values of $\epsilon_n$ are calculated using $P$ as in \eqref{Pdefn} and $\kappa = 0.02$.  A $\star$ indicates best performing agent(s) in the graph as determined in the simulations. 
	}  
	\label{table:5agentgraph_compare}
\end{figure}

\addtocounter{figure}{-1}
\begin{figure}[ht!]	
	\centering
	\includegraphics[width=\GraphWidth]{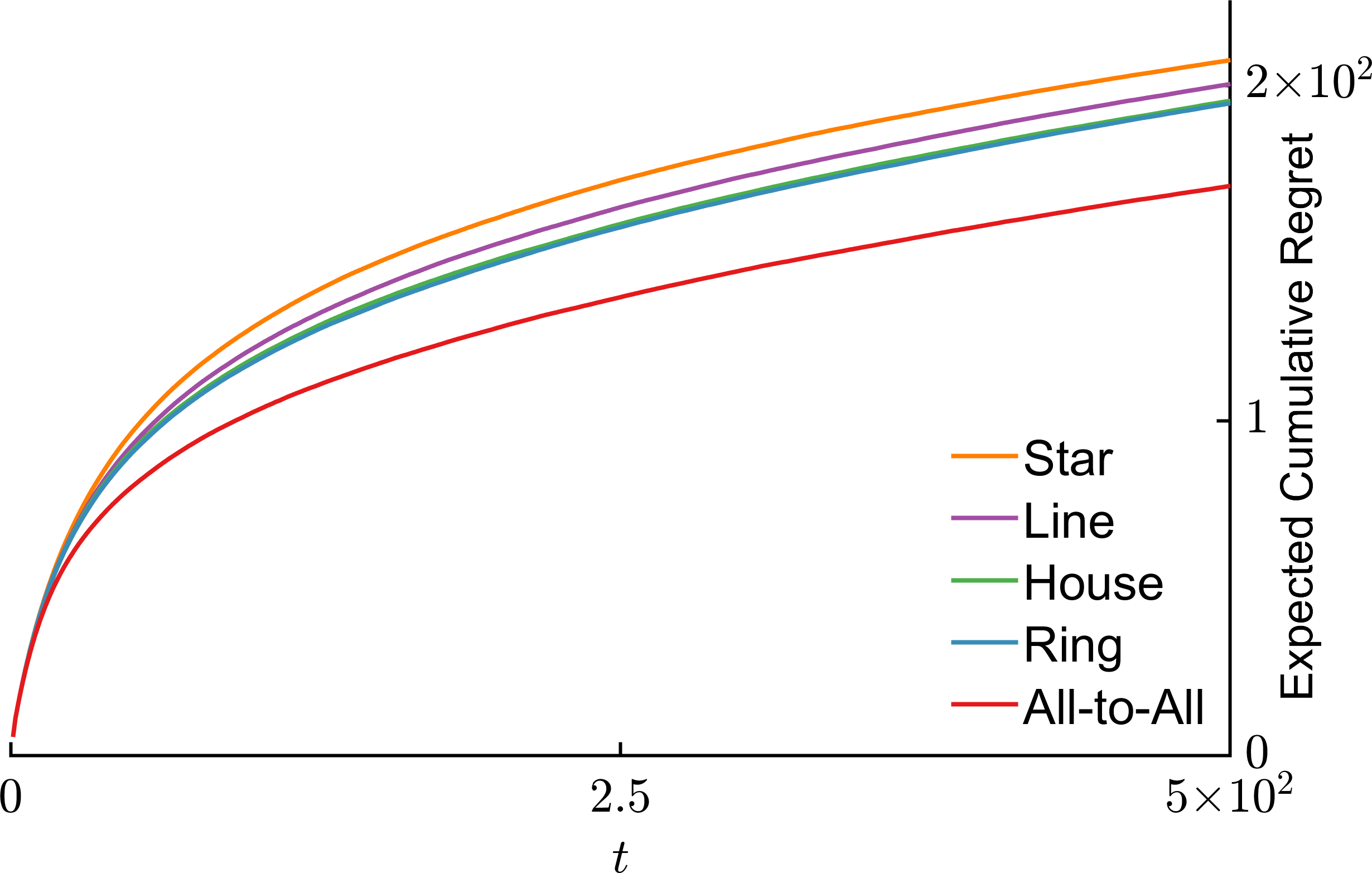}
	\caption{Simulation results of expected cumulative regret of the group for each of the fixed graphs shown in Table \ref{table:5agentgraph_compare}.}
	\label{fig:5AgentGraphCompare} 
\end{figure}

\begin{figure}[ht!]	
	\centering
	\includegraphics[width=\GraphWidth]{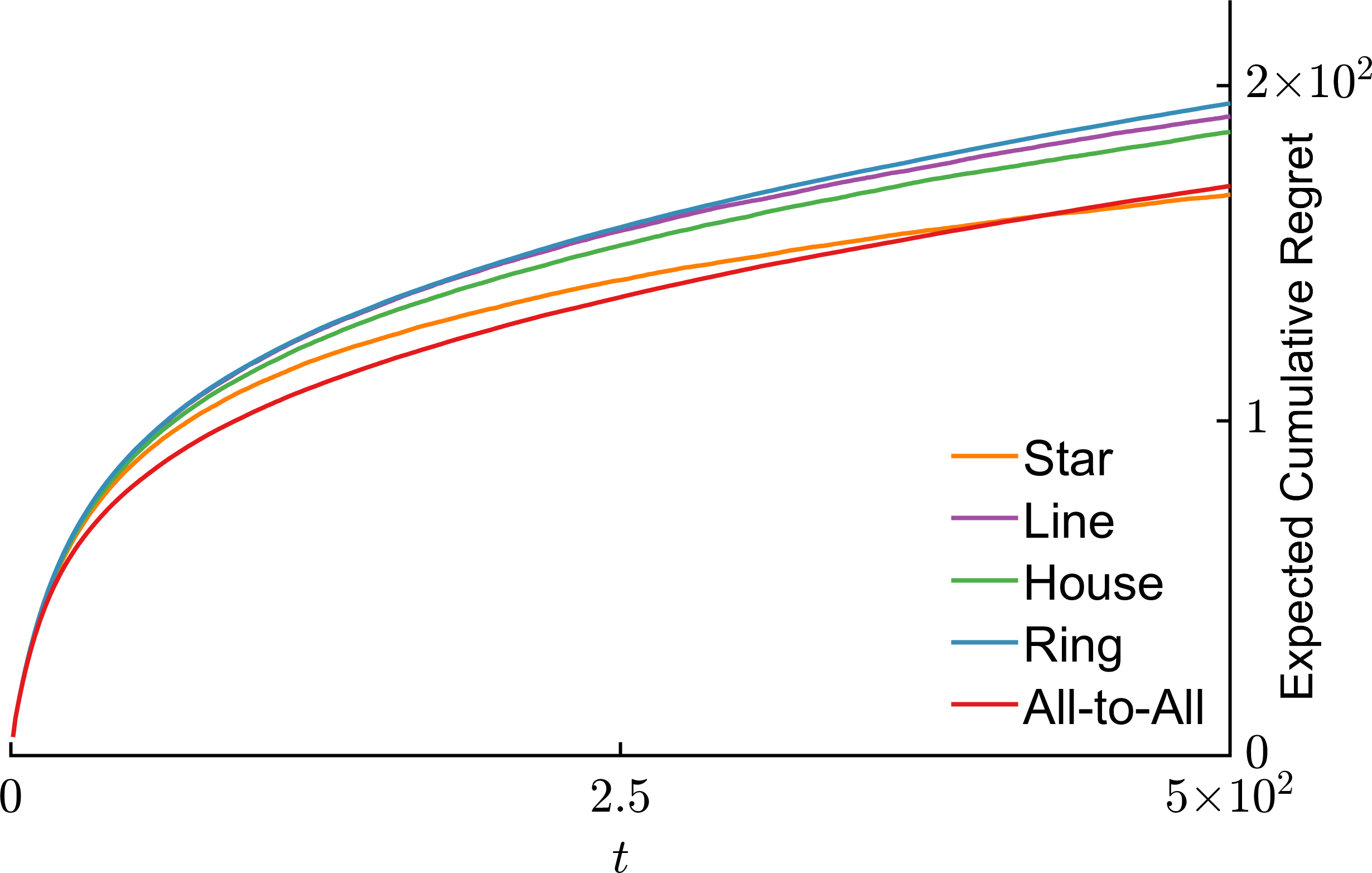}
	\caption{Simulation results of expected cumulative regret of the agent with lowest regret in each of  the fixed graphs shown in Table \ref{table:5agentgraph_compare}.}
	\label{fig:5AgentGraphCompare_BestAgent} 
\end{figure}

\begin{figure}[ht!]	
	\centering
	\includegraphics[width=\GraphWidth]{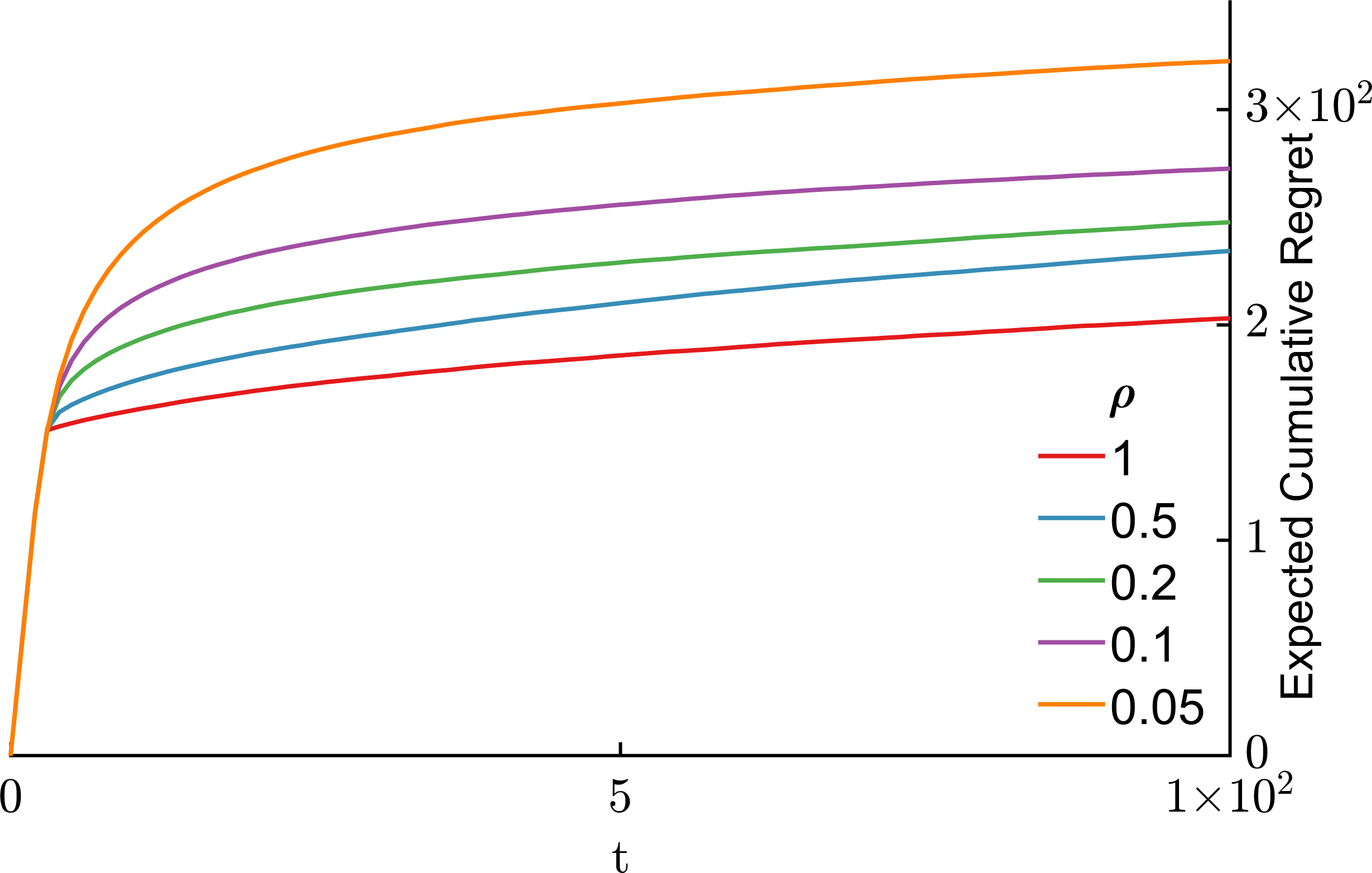}
	\caption{Simulation results of expected cumulative regret of $100$ agents on an Erd{\"o}s-R{\'e}yni random graph for five different values of edge probability $\rho$.
	}
	\label{fig:LargeERGraphs} 
\end{figure}

\begin{example}\label{ex:5node_graphcompare}
Figure \ref{fig:5AgentGraphCompare} compares the expected cumulative regret averaged over all agents in each of the five graphs in Table \ref{table:5agentgraph_compare}, where agents use coop-UCB2.  The value of $\epsilon_n$ is shown in Table \ref{table:5agentgraph_compare} for each graph.  Theorem \ref{thm:regret-coop-ucb2} predicts that graphs with lower $\epsilon_n$ will have lower average expected cumulative regret.  Here we use two arms and $\kappa=0.02$.  Figure \ref{fig:5AgentGraphCompare} verifies this prediction, showing the the ordering of graphs by performance is equal to the ordering of graphs by the graph explore-exploit index $\epsilon_n$.

	Figure \ref{fig:5AgentGraphCompare_BestAgent} compares expected cumulative regret for  best performing agent(s) in each of the five graphs in Table \ref{table:5agentgraph_compare}.  
	The central agent in the star graph outperforms the best agent in the all-to-all graph despite the star graph's poor group performance.  This indicates that the four peripheral agents are doing most of the exploration.  The stark difference in the propensity to explore between the central and peripheral agents in the star graph demonstrates that regret accumulation for different agents could be controlled by design of the communication graph structure.
\end{example}

\begin{example}\label{ex:large_er_graphs}
Figure \ref{fig:LargeERGraphs} compares the average expected cumulative regret of 100 agents using coop-UCB2 (two arms and $\kappa = \frac{d_{\text{max}}}{d_{\text{max}}-1}$) for a range of Erd{\"o}s-R{\'e}yni (ER) random graphs~\cite{bollobas1998random}. 
We simulate five values of the probability $\rho$ of a connection between any two agents, from $\rho = 0.05$ (weakly connected) to $\rho = 1.0$ (fusion center). 
For each $\rho$ we randomly generated $15$ ER graphs. 
We show the results of $2 \! \times \! 10^4$ simulations per graph, or $3\! \times \! 10^5$ simulations per $\rho$.  The plot shows how performance improves as the connection between agents increases.

\end{example}


\section{Final Remarks}\label{FinalRemarks}
We have used a distributed multi-agent MAB problem to explore cooperative decision making under uncertainty for networks of agents.  Each agent makes choices among arms to maximize its own individual reward but cooperates with others in the group by communicating its estimates across the network. 
We considered both an unconstrained reward model, in which agents are not penalized if they choose the same arm at the same time, and a constrained reward model, in which agents that choose the same arm at the same time receive no reward.

We designed an algorithm for distributed cooperative estimation of mean reward at each arm.  Building on this, we designed the coop-UCB2 and coop-UCB2-selective-learning algorithms for the unconstrained and contrained reward models, respectively. These are distributed algorithms that enable agents to leverage the information shared by neighbors in their decision making, without requiring that agents know the network graph structure.   We proved bounds on performance, showing logarithmic expected cumulative group regret close to that of a centralized fusion center, for both reward models. 

From the bounds on regret, we defined a novel graph explore-exploit index and nodal explore-exploit centrality index, which  depend only on the network graph topology. The group index predicts the ordering by performance of network graphs and the nodal index predicts the ordering by performance of the nodes.  

Future research directions include rigorously exploring other communications schemes, which may offer better performance or be better suited to modeling classes of networked systems.  The tradeoff between communication frequency and performance (\cite{madhushani2020observation}) as well as the presence of noisy communications (\cite{SavasACC2017}) will be important considerations.  

\section{Acknowledgements}
The authors thank Tor Lattimore for pointing out an error in a previous version of one of the proofs. 


\appendix

\section{Proof of Proposition~\ref{prop:coop-est}}\label{app:proof-prop1}
We begin with statement (i). 
From~\eqref{nhatdefn} it follows that
\begin{align}
\mathbf{\hat{n}}_i(t) & =P^t \mathbf{\hat{n}}_i(0) + \sum_{\tau = 1}^{t} P^{t-\tau+1} \boldsymbol{\xi}_i(\tau) \nonumber \\
&=\sum_{\tau = 0}^{t}  \Big[ \frac{1}{M} \mathbf{1}_M \mathbf{1}_M^\top \boldsymbol{\xi}_i(\tau) + \sum_{p=2}^{M} \lambda_p^{t-\tau+1} \mathbf{u}_p {\mathbf{u}_p}^\top \boldsymbol{\xi}_i(\tau) \Big] \nonumber\\
&=  n_i^{\text{cent}}(t) \mathbf{1}_M + \sum_{\tau = 1}^{t} \sum_{p=2}^{M} \lambda_p^{t-\tau+1} \mathbf{u}_p {\mathbf{u}_p}^\top \boldsymbol{\xi}_i(\tau). \label{solndecomp}
\end{align}

For (i), we bound the $k$-th entry of the second term of~\eqref{solndecomp}:
\begin{align}
\sum_{\tau = 1}^{t} \sum_{p=2}^M  \lambda_p^{t-\tau+1}  \big( \mathbf{u}_p {\mathbf{u}_p}^\top \boldsymbol{\xi}_i(\tau) \big)_k  \!
&\leq \!   \sum_{\tau = 1}^{t} \sum_{p=2}^M |\lambda_p^{t-\tau+1}| \| \mathbf{u}_p\|_2^2   \| \boldsymbol{\xi}_i(\tau) \|_2  \nonumber \\
& \leq   \sqrt{M} \sum_{\tau = 1}^{t} \sum_{p=2}^M  |\lambda_p^{t-\tau+1}| \le \epsilon_n.  \nonumber
\end{align}

To prove statement (ii), let $\nu_{pwi}(\tau) \!=\! \sum_{j=1}^M u_p^j u_w^j \xi_i^j(\tau)$ and then 
\begin{align}
\sum_{\tau=1}^{t} &\sum_{j=1}^M \left(\sum_{p=1}^M \lambda_p^{t-\tau+1}  u_p^k u_p^j  \right)^2 \xi_i^j(\tau) \nonumber \\
& = \sum_{\tau=1}^t \sum_{p=1}^M \sum_{w=1}^M (\lambda_p \lambda_w)^{t-\tau+1} u_p^k u_w^k \sum_{j=1}^M u_p^j u_w^j \xi_i^j(\tau) \nonumber \\
& = \sum_{\tau=1}^t \sum_{p=1}^M \sum_{w=2}^M (\lambda_p \lambda_w)^{t-\tau+1} u_p^k u_w^k \nu_{pwi}(\tau) \nonumber \\
& \quad \quad + \frac{1}{M} \sum_{\tau=1}^t \sum_{p=1}^M  \sum_{j=1}^M \lambda_p^{t-\tau+1} u_p^k u_p^j \xi_i^j(\tau) \nonumber \\
& = \sum_{\tau=1}^t \sum_{p=1}^M \sum_{w=2}^M (\lambda_p \lambda_w)^{t-\tau+1} u_p^k u_w^k \nu_{pwi}(\tau) +\frac{1}{M} \hat{n}_i^k(t). \label{covssoln}
\end{align}
This establishes (ii) since for the first term of \eqref{covssoln}:
\begin{align*}
\sum_{\tau=1}^t  (\lambda_p \lambda_w)^{t-\tau+1} u_p^k u_w^k \nu_{pwi}(\tau) 
&  
\leq \sum_{\tau=1}^t  |(\lambda_p \lambda_w)^{t-\tau+1} || u_p^k u_w^k \nu_{pwi}(\tau) | \nonumber \\ 
\leq  \sum_{\tau = 0}^{t-1}   |\lambda_p \lambda_w|^{t-\tau+1} a_{pw}(k) 
&
\leq   \frac{|\lambda_p \lambda_w|}{1-|\lambda_p \lambda_w|}  a_{pw}(k). \label{cov1stmid}
\end{align*}


\section{Proof of Theorem~\ref{thm:EstDevBoudnsCondensed}}\label{app:proof-thm-1}

We begin by noting that $\hat{s}_i^k(t)$ can be decomposed as 
	\begin{equation}
		\hat{s}_i^k(t) = \sum_{\tau=1}^{t} \sum_{p=1}^{M} \lambda_p^{t-\tau+1} \sum_{j=1}^M u_p^k u_p^j r_i^j(\tau) \xi_i^j(\tau). \label{eqn:s_i_hat_defn}
	\end{equation}
	Let $\hat{s}_i^{kp}(t) =  \sum_{\tau=1}^{t} \lambda_p^{t-\tau+1} \sum_{j=1}^M u_p^k u_p^j r_i^j(\tau) \xi_i^j(\tau)$.  Then, 
	\begin{align}
		\sum_{p=1}^M \hat{s}_i^{kp}(t)  =   \sum_{p=1}^{M} \sum_{j=1}^M \lambda_p u_p^k u_p^j r_i^j(t) \xi_i^j(t) + \sum_{p=1}^M \lambda_p \hat{s}_i^{kp}(t-1). \label{eqn:shat_modaldecomp}
	\end{align}

	It follows from \eqref{eqn:s_i_hat_defn} and \eqref{eqn:shat_modaldecomp} that for any $\Theta > 0$
	\begin{align*}
		\mathbb{E}& \left[\expp{\Theta \hat{s}_i^k(t)} \middle \vert \mc F_{t-1} \right] =
		\mathbb{E}\left[\expp{\Theta \sum_{p=1}^M \hat{s}_i^{kp}(t)} \middle \vert \mc F_{t-1} \right] 	
		\\
		& = \mathbb{E} \left[ \expp{\Theta \sum_{p=1}^M \lambda_p \sum_{j=1}^M u_p^k u_p^j r_i^j(t) \xi_i^j(t)}  \middle \vert \mc F_{t-1} \right] K_{(t-1)} \\
		%
		& = \prod_{j=1}^M \mathbb{E} \left[ \expp{\Theta \sum_{p=1}^M \lambda_p  u_p^k u_p^j r_i^j(t) \xi_i^j(t)}  \middle \vert \mc F_{t-1} \right] K_{(t-1)} 
				\end{align*}
		\begin{align*}
		%
		& = \expp{\sum_{j=1}^M \phi_i \left( \Theta \sum_{p=1}^M \lambda_p u_p^k u_p^j \xi_i^j(t) r_i^j(t)  \right)  } K_{(t-1)} \\
				%
		& = \expp{\sum_{j=1}^M \phi_i \left( \Theta \sum_{p=1}^M \lambda_p u_p^k u_p^j r_i^j(t)  \right) \xi_i^j(t) } {K_{(t-1)},}\\
	\end{align*}
	\[
	{K_{(t-1)} = \expp{\Theta \sum_{p=1}^M \lambda_p \hat{s}_i^{kp}(t-1)},}
	\]
	and the second-to-last equality follows since,
	conditioned on $\mc F_{t-1}$, $ \xi_i^j(t)$ is 
	deterministic 
	and $r_i^j(t)$ are i.i.d.~for each $j \in \until{M}$.  The last equality follows since $\xi_i^j(t)$ is binary and the two expressions are the same for $\xi_i^j(t)\in \{0,1\}$. Therefore, 
\[
		\mathbb{E} \Bigg[ \! \exp \bigg(\Theta \!\! \sum_{p=1}^M \hat{s}_i^{kp}(t) - \! \sum_{j=1}^M \phi_i \left( \Theta \!\! \sum_{p=1}^M \! \lambda_p u_p^k u_p^j r_i^j(t)  \right) 
		\xi_i^j(t)\bigg)
			\bigg| \mc F_{t-1} \! \Bigg] \! = \! {K_{(t-1)}.}
			\]
	%
	Using the above argument recursively with $s_i^k(0)=0$, we obtain
		\[
		\mathbb{E} \Bigg[\exp \Bigg( \Theta  \hat{s}_i^k(t)  
		- \sum_{\tau=1}^t \sum_{j=1}^M \phi_i \left(\Theta \sum_{p=1}^M \lambda_p^{t-\tau+1} u_p^k u_p^j r_i^j(\tau) \right) 
		\xi_i^j(\tau) \Bigg ) \Bigg] =1. 
	\]
	For sub-Gaussian random variables $\phi_i(\beta) \leq \beta m_i + \frac{1}{2} \sigma_g^2 \beta^2$, thus
	\begin{align}
		1 &=\mathbb{E} \Bigg [ \! \exp \Bigg( \Theta \!\left(\hat{s}_i^k\!(t) \!-\! m_i \hat{n}_i^k\!(t)\right) \label{eqn:CGF_sub}\\
		& \qquad \qquad - \frac{\sigma_g^2}{2}\! \sum_{\tau=1}^{t} \sum_{j=1}^M \!\left(\! \Theta\! \sum_{p=1}^M \lambda_p^{t-\tau+1} \! u_p^k u_p^j  \right)^2  \! \! \xi_i^j(\tau) \Bigg) \Bigg ] \nonumber \\
		&\ge \mathbb{E} \Bigg [ \expp{ \! \Theta \!\left(\hat{s}_i^k\!(t) \!-\! m_i \hat{n}_i^k\!(t)\right) \!-\! \frac{\sigma_g^2 \Theta^2}{2M} \left(\hat{n}_i^k(t) + \epsilon_c^k \right)}\Bigg ], \nonumber 
	\end{align}
	where the last inequality follows from the second statement of Proposition~\ref{prop:coop-est}. 	
	Now using the Markov inequality, we obtain
	\begin{align}
		&e^{-a} \! \geq \mathbb{P} \! \left( \! \expp{ \! \Theta \!\left(\hat{s}_i^k\!(t) \!-\! m_i \hat{n}_i^k\!(t)\right) \!-\! \frac{\sigma_g^2 \Theta^2}{2M} \left(\hat{n}_i^k(t) + \epsilon_c^k \right)} \! \! \geq \! e^{a} \! \right) \nonumber \\
		&= \mathbb{P} \Bigg( \frac{\hat{s}_i^k(t) - m_i \hat{n}_i^k(t)}{\left( \frac{1}{M}\left(\hat{n}_i^k(t) + \epsilon_c^k \right)  \right)^{\frac{1}{2}}}  \geq  \frac{a}{\Theta} \left( \frac{1}{M}\left(\hat{n}_i^k(t) + \epsilon_c^k \right)  \right)^{-\frac{1}{2}} \nonumber \\
		& \qquad \qquad \qquad \qquad  + \frac{\sigma_g^2 \Theta}{2}\left( \frac{1}{M}\left(\hat{n}_i^k(t) + \epsilon_c^k \right)  \right)^{\frac{1}{2}} \Bigg ).
		\label{FromMarkovIn}
	\end{align}
	
Random variable $\hat n_i^k(t)$ on the right  of \eqref{FromMarkovIn} 
depends on the random variable on the left. So, we use union bounds on $\hat n_i^k(t)$ to obtain the 
concentration inequality. Consider an exponentially increasing sequence of time indices $\setdef{(1+\eta)^{h-1}}{h \in \until{D}}$, where $D = \left \lceil \frac{\lnn{t+ \epsilon_n}}{\lnn{1 + \eta}} \right \rceil$ and $\eta >0$. 
For every $h \in \{1,\dots,D\}$, define
	\begin{equation}
		\Theta_h = \frac{1}{\sigma_g} \sqrt{\frac{2 a M}{(1+\eta)^{h-\frac{1}{2}} + \epsilon_c^k}}.
	\end{equation}
	Thus, if $(1+\eta)^{h-1} \leq \hat{n}_i^k(t) \leq (1+\eta)^{h}$, then
	\begin{align}
		&\frac{a}{\Theta_h} \left( \frac{1}{M}\left(\hat{n}_i^k(t) + \epsilon_c^k \right)  \right)^{-\frac{1}{2}}  + \frac{ \sigma_g^2 \Theta_h}{2}\left( \frac{1}{M}\left(\hat{n}_i^k(t) + \epsilon_c^k \right)  \right)^{\frac{1}{2}} \nonumber \\
		& = \sigma_g \sqrt{\frac{a}{2}} \left(\! \left( \frac{(1+ \eta)^{h-\frac{1}{2}} + \epsilon_c^k }{\hat{n}_i^k(t) + \epsilon_c^k} \right)^{\frac{1}{2}} \!\! + \left( \frac{\hat{n}_i^k(t) + \epsilon_c^k}{(1+ \eta)^{h-\frac{1}{2}} + \epsilon_c^k } \right)^{\frac{1}{2}}   \! \right) \nonumber \\
		& \le \sigma_g \sqrt{\frac{a}{2}} \left(\! \left( \frac{(1+ \eta)^{h-\frac{1}{2}} }{\hat{n}_i^k(t) } \right)^{\frac{1}{2}} \!\! + \left( \frac{\hat{n}_i^k(t) }{(1+ \eta)^{h-\frac{1}{2}}} \right)^{\frac{1}{2}}   \! \right) \nonumber \\
		&\leq \sigma_g \sqrt{\frac{a}{2}} \left(\! (1+\eta)^{\frac{1}{4}} +  (1+\eta)^{-\frac{1}{4}}   \! \right),
		\label{Finding_delta}
	\end{align}
	where the second-to-last inequality follows from the fact that for $a,b >0$, the function $\epsilon \mapsto \sqrt{\frac{a + \epsilon}{b+\epsilon}} + \sqrt{\frac{b + \epsilon}{a+\epsilon}}$ with domain $\real_{\ge 0}$ is monotonically non-increasing, and
	the last inequality follows from the fact that for $\eta>0$, the function $ x \mapsto \sqrt{\frac{(1+ \eta)^{h-\frac{1}{2}}}{x}} \ + \sqrt{\frac{x}{(1+ \eta)^{h-\frac{1}{2}}}} $ with domain $ [(1+\eta)^{h-1}, (1+\eta)^h]$ achieves its maximum at either of the boundaries. {Applying union bounds on $D$ possible values of $h$ and using \eqref{Finding_delta} for $(1+\eta)^{h-1} \le \hat n_i^k(t) \le (1+\eta)^h$, from \eqref{FromMarkovIn} we get}
	\begin{align*}
		&\mathbb{P} \left( \frac{\hat{s}_i^k(t) - m_i \hat{n}_i^k(t)}{\left(\frac{1}{M} \left(\hat{n}_i^k(t) + \epsilon_c^k\right)\right)^{\frac{1}{2}}} > \sigma_g \sqrt{\frac{a}{2}} \left( (1+\eta)^{\frac{1}{4}} + (1+\eta)^{-\frac{1}{4}}   \right)  \right) \\
		& \le \sum_{h=1}^D \mathbb{P}\; \Vast( \frac{\hat{s}_i^k(t) - m_i \hat{n}_i^k(t)}{\left(\frac{1}{M} \left(\hat{n}_i^k(t) + \epsilon_c^k\right)\right)^{\frac{1}{2}}} \! > \! \frac{a}{\Theta_h}\left(\frac{1}{M} \left(\hat{n}_i^k(t) + \epsilon_c^k\right)\right)^{-\frac{1}{2}}  \\
		&\qquad \qquad +\frac{ \sigma_g^2 \Theta_h}{2} \left(\frac{1}{M} \left(\hat{n}_i^k(t) + \epsilon_c^k\right)\right)^{\frac{1}{2}}\\
		&\qquad \quad   \; \&\;  (1+\eta)^{h-1} \le \hat n_i^k(t) + \epsilon_c^k< (1+\eta)^h   \Vast) \le D e^{-a}. 
	\end{align*}

	Setting $ \sigma_g \sqrt{\frac{a}{2}} \left( (1+\eta)^{\frac{1}{4}} + (1+\eta)^{-\frac{1}{4}} \right) = \delta$  yields 
	\begin{align*}
		&\mathbb{P} \left( \frac{\hat{s}_i^k(t) - m_i \hat{n}_i^k(t)}{\left(\frac{1}{M} \left(\hat{n}_i^k(t) + \epsilon_c^k\right)\right)^{\frac{1}{2}}} > \delta \right)  \leq D \expp{\frac{-2 \delta^2}{\sigma_g^2 \left( (1+\eta)^{\frac{1}{4}}\! +\! (1+\eta)^{-\frac{1}{4}} \right)^2 }}.
	\end{align*}
	{It can be verified using Taylor series expansion that} 
	\begin{equation*}
		\frac{4}{\left( (1+\eta)^{\frac{1}{4}} + (1+\eta)^{-\frac{1}{4}} \right)^2} \geq 1-\frac{\eta^2}{16}.
	\end{equation*}
	Therefore,	it holds that 
	\begin{align*}
		&\mathbb{P} \left( \frac{\hat{s}_i^k(t) - m_i \hat{n}_i^k(t)}{\left(\frac{1}{M} \left(\hat{n}_i^k(t) + \epsilon_c^k\right)\right)^{\frac{1}{2}}} > \delta \right) \leq D \expp{\frac{-\delta^2}{2\sigma_g^2}\left( 1-\frac{\eta^2}{16} \right) } \\
		& \qquad \qquad  = \Bigg\lceil \frac{\lnn{t + \epsilon_n}}{\lnn{1+\eta}} \Bigg\rceil  \expp{\frac{-\delta^2}{2\sigma_g^2}\left( 1-\frac{\eta^2}{16} \right) }.
	\end{align*}
	
\section{Pseudocode for coop-UCB2}\label{app:pseudocode}

\IncMargin{.3em}
\begin{algorithm}[ht!]
  {\footnotesize
   \SetKwInOut{Input}{Input}
   \SetKwInOut{Set}{Set}
   \SetKwInOut{Title}{Algorithm}
   \SetKwInOut{Require}{Require}
   \SetKwInOut{Output}{Output}
   \Input{arms $\until{N}$, agents $\until{M}$\;} 
   
      \Input{parameters $\sigma_g>0$, $\eta>0$, $\gamma>1$, function $f(t)$\;}
   \Output{allocation sequence $i^k(t), t\in \until{T}, k \in \until{M}$\;}

   \medskip

\nl {\bf set} $\hat n_i^k \leftarrow 0, \hat s_i^k \leftarrow 0$, $i\in \until{N},  k \in \until{M}$\;
\smallskip

   \nl \For {$t \in \until{T}$ }{

\If{$t\le N$}{ 
\emph{\% Initialization} 

\smallskip 

\nl \For{each agent $k \in \until{M}$ \smallskip }{ $i^k(t) \leftarrow t$ \; \smallskip

collect reward $r^k(t)$ \;

}

\smallskip 
}

\nl \Else {

\nl \For{each agent $k \in \until{M}$ \smallskip }{

\emph{\% select arm with maximum $Q_i^k$}
\smallskip

\For{ each arm $i \in \until{N}$\smallskip }{
$Q_i^k \leftarrow \frac{\hat s_i^k}{\hat n_i^k} + \sigma_g \; \sqrt[]{\frac{2\gamma}{G(\eta)} \cdot \frac{\hat{n}_i^{k} +  f(t-1)}{M\hat{n}_i^{k}}\cdot\frac{ \lnn{t-1}}{\hat{n}_i^{k}}}$\; \smallskip
}

$ i^k(t) \leftarrow \argmax\setdef{Q_i^k}{i \in \until{N}}$ \;

\smallskip

collect reward $r^k(t)$ \; 

}

\smallskip 

}

\nl \For {$i \in \until{N}$ \smallskip }{

\nl update $\bs{\hat n_i}$ and $\bs{\hat s_i}$ using~\eqref{nhatdefn} and~\eqref{shatdefn}\;
\smallskip 
}

}

    \caption{\textit{coop-UCB2}}
  \label{algo:coop-ucb2}}
\end{algorithm} 
\DecMargin{.3em}


\section{Proof of Theorem~\ref{thm:regret-coop-ucb2}}\label{app:proof-thm2}
We proceed similarly to~\cite{PA-NCB-PF:02}.  The number of selections of a suboptimal arm $i$ by all agents until time $T$ is 
\begin{align}
& \sum_{k=1}^M  n_i^k(T) 
\leq \sum_{k=1}^M (t_k^\dag-1) + \sum_{k=1}^M \sum_{t= t_k^\dag}^T  \mathds{1}(Q_i^k(t-1) \geq Q_{i^*}^k(t-1))\nonumber \\
& \leq  A + \! \sum_{k=1}^M \left( (t_k^\dag-1) 
\!+\! \sum_{t=t_k^\dag}^T \mathds{1}(Q_i^k(t-1) \geq Q_{i^*}^k(t-1), M {n}^{\text{cent}}_i \geq A)\right)  \label{suboptimal-samples}
\end{align}
where $A >0$ is a constant that will be chosen later.  

At a given time $t+1$ an individual agent $k$ will choose a suboptimal arm only if 
$ Q_i^k(t) \geq Q_{i^*}^k(t)$. 
For this condition to be true at least one of the following three conditions must hold:
\begin{align}
\hat{\mu}_{i^*}(t) &\leq m_{i^*} - C_{i^*}^k(t) \label{1stcond} \\
\hat{\mu}_{i}(t) &\geq m_{i} + C_{i}^k(t) \label{2ndcond} \\
m_{i^*} &< m_{i} + 2 C_{i}^k(t). \label{3rdcond} 
\end{align}
We bound the probability that \eqref{1stcond} and \eqref{2ndcond} hold using Theorem~\ref{thm:EstDevBoudnsCondensed}: 
\begin{align*}
&\mathbb{P}\left( \eqref{1stcond} \textrm{ holds } | \, t \geq  t_k^\dagger \right) \\
& \qquad = \mathbb{P} \!\left(\! \frac{\hat{s}_{i}^k - m_{i} \hat{n}_{i}^k}{\sqrt{ \frac{1}{M} \left( \hat{n}_i^{k}(t) +  f(t)\right) }} \! \geq \! \sigma_g \sqrt{\frac{2 \gamma \lnn{t}}{G(\eta)}} \, \Bigg | \, t \geq t_k^\dagger \right)\\
& \qquad \leq \mathbb{P} \!\left(\! \frac{\hat{s}_{i}^k - m_{i} \hat{n}_{i}^k}{\sqrt{ \frac{1}{M} \left( \hat{n}_i^{k}(t) +  \epsilon_c^k\right) }} \! \geq \! \sigma_g \sqrt{\frac{2 \gamma \lnn{t}}{G(\eta)}} \, \Bigg | \, t \geq t_k^\dagger \right)\\
& \qquad \leq \left( \frac{\lnn{t}}{\lnn{1+\eta}} + \frac{\lnn{1+\epsilon_n}}{\lnn{1+\eta}} + 1\right) \frac{1}{t^\gamma},
\end{align*}
\begin{equation*}
\mathbb{P}\left( \eqref{2ndcond} \textrm{ holds } | t \geq  t_k^\dagger \right) \leq \left( \frac{\lnn{t}}{\lnn{1+\eta}} + \frac{\lnn{1+\epsilon_n}}{\lnn{1+\eta}} + 1\right) \frac{1}{t^\gamma}.
\end{equation*}

We now examine the event \eqref{3rdcond}. 
\begin{align}
m_{i^*} & < m_i + 2C_i^k(t) \nonumber \\ 
\implies \hat{n}_i^k(t)^2 \frac{\Delta_i^2M G(\eta)}{8 \sigma_g^2} &- \gamma \hat{n}_i^k(t)\ln(t) -  \gamma f(t) \ln(t) < 0. \label{eqn:3rdcondquad}
\end{align}
The quadratic equation~\eqref{eqn:3rdcondquad} can be solved to find its roots, and if $\hat{n}_i(t)$ is greater than the larger root the inequality will never hold.
Solving the quadratic equation~\eqref{eqn:3rdcondquad}, we obtain that event~\eqref{3rdcond} does not hold if
\begin{align*}
\hat{n}_i^k(t) &\geq \frac{ 4 \sigma_g^2 \gamma \ln(t)}{\Delta_i^2MG(\eta)} \!+\! \sqrt{\Big(\frac{4 \gamma \sigma_g^2 \ln(t)}{\Delta_i^2MG(\eta)}\Big)^2 \!+  \frac{8 \sigma_g^2f(t) \gamma \ln(t)}{\Delta_i^2 MG(\eta)}  } \\
& = \frac{4 \sigma_g^2 \gamma \ln t}{\Delta_i^2 MG(\eta)} \left( 1 + \sqrt{1 + \frac{\Delta_i^2 MG(\eta)}{2 \sigma_g^2 \gamma } \frac{f(t)}{\ln t}} \right).
\end{align*}

Now, we set $A = \Big\lceil M \epsilon_n + \frac{4 \sigma_g^2 \gamma \ln T}{\Delta_i^2 G(\eta)} \big( 1 + \sqrt{1 + \frac{ \Delta_i^2 M G(\eta)}{2 \gamma \sigma_g^2} \frac{f(T)}{\ln T}} \big)
\Big\rceil$. It follows from monotonicity of $f(t)$ and $\ln(t)$ and statement (i) of Proposition~\ref{prop:coop-est}
that event~\eqref{3rdcond} does not hold if $M n^{\text{cent}}_i(t) > A$.  

  Therefore, from \eqref{suboptimal-samples} we see that
\begin{align*} 
	&\sum_{k=1}^M  \mathbb{E} \left[ n_i^k(T)\right] \leq \bar{A} +  \sum_{k=1}^M (t_k^\dag-1)\\
	& \quad \qquad + \!\frac{2}{\lnn{1\!+\!\eta}} \sum_{k=1}^M \sum_{t=t_k^\dagger}^T \! \left( \frac{ \lnn{t}}{t^\gamma} \!+\! \frac{\lnn{(1\!+\!\epsilon_n)(1\!+\!\eta)}}{t^\gamma} \right) \\
	&
	\leq \bar{A}+ \sum_{k=1}^M  (t_k^\dagger-1) 
	+ \!\frac{2M}{\lnn{1\!+\!\eta}} \sum_{t=1}^T \! \left( \frac{ \lnn{t}}{t^\gamma} \!+\! \frac{\lnn{(1\!+\!\epsilon_n)(1\!+\!\eta)}}{t^\gamma} \right) \\
	& 
	\leq \bar{A} \!+\! \sum_{k=1}^M  (t_k^\dagger\!-\!1) + \!\frac{2M}{\lnn{1\!+\!\eta}}  \! \Big( \frac{1}{(\gamma - 1)^2} 
	\!+\! {\frac{\gamma \lnn{1\!+\!\epsilon_n)(1\!+\!\eta)}}{\gamma - 1} +1}\Big), 
  \end{align*}  
  where $\bar{A} = \max\{M,A\}$  is chosen 
  to account for the $M$ selections of the $i$-th arm during the initialization phase. 
  
  \section{Pseudocode for coop-UCB2-selective-learning}\label{app:pseudocode2}

\IncMargin{.3em}
\begin{algorithm}[ht!]
  {\footnotesize
   \SetKwInOut{Input}{Input}
   \SetKwInOut{Set}{Set}
   \SetKwInOut{Title}{Algorithm}
   \SetKwInOut{Require}{Require}
   \SetKwInOut{Output}{Output}
   \Input{arms $\until{N}$, agents $\until{M}$\;} 
   
      \Input{parameters $\sigma_g>0$, $\eta>0$, $\gamma>1$, function $f(t)$\;}
   \Output{allocation sequence $i^k(t), t\in \until{T}, k \in \until{M}$\;}

   \medskip

\nl {\bf set} $\hat n_i^k \leftarrow 0, \hat s_i^k \leftarrow 0$, $i\in \until{N},  k \in \until{M}$\;
\smallskip

   \nl \For {$t \in \until{T}$ }{

\If{$t\le N$}{ 
\emph{\% Initialization} 

\smallskip 

\nl \For{each agent $k \in \until{M}$ \smallskip }{ $i^k(t) \leftarrow (t-1+k) \mod N$ \; \smallskip

collect reward $r^k(t)$ \;

}

\smallskip 
}

\nl \Else {

\nl \For{each agent $k \in \until{M}$ \smallskip }{
\For{ each arm $i \in \until{N}$\smallskip }{

$Q_i^k \leftarrow \frac{\hat s_i^k}{\hat n_i^k} + \sigma_g \; \sqrt[]{\frac{2\gamma}{G(\eta)} \cdot \frac{\hat{n}_i^{k} +  f(t-1)}{M\hat{n}_i^{k}}\cdot\frac{ \lnn{t-1}}{\hat{n}_i^{k}}}$\;
}
\smallskip

\emph{\% Compute descending sort indices for $Q_i^k$} \smallskip

$I_i^k  \leftarrow \mathrm{sort\_index}(\setdef{Q_i^k}{i\in\until{N}}, \text{`descend'})$\; \smallskip 

\emph{\% Estimate $k$-best arms} \smallskip

$\mc O_k \leftarrow \{I_1^k, \ldots, I_k^k\}$ \; \smallskip

\emph{\% select the worst arm from $k$-best arms} \smallskip

\For{ each arm $i \in \mc O_k$\smallskip }{
$W_i^k \leftarrow \frac{\hat s_i^k}{\hat n_i^k} - \sigma_g \; \sqrt[]{\frac{2\gamma}{G(\eta)} \cdot \frac{\hat{n}_i^{k} +  f(t-1)}{M\hat{n}_i^{k}}\cdot\frac{ \lnn{t-1}}{\hat{n}_i^{k}}}$\; 

}
\smallskip 

$ i^k(t) \leftarrow \argmin\setdef{W_i^k}{i \in \mc O_k}$ \;

\smallskip

collect reward $r^k(t)$ \; 

}

\smallskip 

}

\nl \For {$i \in \until{N}$ \smallskip }{

\nl update $\bs{\hat n_i}$ and $\bs{\hat s_i}$ using~\eqref{nhatdefn} and~\eqref{shatdefn}\;
\smallskip 
}

}

    \caption{\textit{coop-UCB2-selective-learning}}
  \label{algo:coop-ucb2-selective}}
\end{algorithm} 
\DecMargin{.3em}

  \section{Proof of Theorem~\ref{thm:incorrect-selections}}\label{app:proof-thm3}
  	We begin by noting that
	\begin{align}
&	\sum_{k \neq k^i}  n_i^k(T)  = \sum_{k \neq k^i}  \sum_{t=1}^{T} \mathds{1} \left\{ i^k(t) = i \right\} \nonumber \\
	& = \sum_{k \neq k^i}  \sum_{t=1}^{T} {\left( \mathds{1} \left\{ i^k(t) = i, m_i < m_{b^k} \right\} 
	+ 
	\mathds{1} \left\{ i^k(t) = i, m_i \geq m_{b^k} \right\} \right)} \nonumber \\
	& \leq A \!+ \! \sum_{k=1}^M (t_k^\dag -1) 
	+ \sum_{k \neq k^i}  \sum_{t=t_k^\dag}^{T} \mathds{1} \left\{ i^k(t) = i, m_i < m_{b^k}, Mn_i^{\text{cent}}(t) \geq A \right\} \nonumber \\
	& \qquad + \sum_{k \neq k^i} \sum_{t=t_k^\dag}^{T} \mathds{1} \left\{ i^k(t) = i, m_i \geq m_{b^k}, Mn_i^{\text{cent}}(t) \geq A \right\}, \label{eqn:col_mainsplitting}
	\end{align}
	where $A$ is a constant that will be chosen later.  
	In the case where $m_i < m_{b^k}$, agent $k$ picking arm $i$ implies that there exists an arm $j \in \mathcal{O}_{k}^*$ such that $j \notin \mathcal{O}_{k}(t)$.  Therefore, the following holds:
	\begin{align}
		&\sum_{k \neq k^i}  \sum_{t=t_k^\dag}^{{T}} \mathds{1} \left\{ i^k(t) = i, m_i < m_{b^k}, Mn_i^{\text{cent}}(t) \geq A \right\} \nonumber \\
		& \leq \sum_{k \neq k^i}  \sum_{t_k^\dag-1}^{T-1} \mathds{1} \big\{ Q_i^k(t) \geq Q_j^k(t), \text{for some } j \in \mc O_k^*\setminus \mc O_k(t),  \nonumber \\ 
		& \qquad \qquad \qquad \qquad \qquad  m_i < m_{b^k}, Mn_i^{\text{cent}}(t) \geq A \big\} \nonumber \\
		& \leq \sum_{k \neq k^i}  \sum_{t=t_k^\dag-1}^{T} \sum_{j \in \mathcal{O}_{k}^*} \mathds{1} \big\{ Q_i^k(t) \geq Q_j^k(t),  m_i < m_{b^k}, Mn_i^{\text{cent}}(t) \geq A \big\} \nonumber 
				\end{align}
\begin{align}
		& \leq \sum_{k \neq k^i}   \sum_{j \in \mathcal{O}_{k}^*} \sum_{t=t_k^\dag}^{T} \mathds{1} \big\{ Q_i^k(t) \geq Q_j^k(t),  m_i < m_{b^k}, Mn_i^{\text{cent}}(t) \geq A \big\} \nonumber .
	\end{align}

	As in Theorem \ref{thm:regret-coop-ucb2}, $Q_i^k(t-1) \geq Q_j^k(t-1)$ implies that at least one of the following three conditions must hold for any $j \in O^*_{k}$:
	\begin{align}
		\hat{\mu}_{j}(t) &\leq m_{j} - C_{j}^k(t) \label{1stcond_col} \\
		\hat{\mu}_{i}(t) &\geq m_{i} + C_{i}^k(t) \label{2ndcond_col} \\
		m_{j} &< m_{i} + 2 C_{i}^k(t). \label{3rdcond_col} 
	\end{align}
	The first two equations are bounded using Theorem \ref{thm:EstDevBoudnsCondensed} as in the proof of Theorem~\ref{thm:regret-coop-ucb2}.  The third equation is equivalent to 
	\begin{equation*}
		2 C_i^k(t) > \Delta_{j,i} > \Delta_{\text{min}},
	\end{equation*}
	which, as in the proof of Theorem \ref{thm:regret-coop-ucb2}, does not hold if
	\begin{equation*}
		n_i^k(t) >   \frac{4 \sigma_g^2 \gamma  }{\Delta_{\min}^2 G(\eta)} \left( 1 + \sqrt{1  + \frac{\Delta_{\min}^2 MG(\eta)}{2 \sigma_g^2 \gamma} \frac{f(T)}{\ln T}} \right)  \ln T.
	\end{equation*}
	Therefore, for
	\begin{equation*}
		A  =  \left \lceil  M \epsilon_n  +  \frac{4 \sigma_g^2 \gamma  }{\Delta_{\min}^2 G(\eta)}  \left( 1 + \sqrt{1  +  \frac{\Delta_{\min}^2 MG(\eta)}{2 \sigma_g^2 \gamma} \frac{f(T)}{\ln T}} \right)  \ln T \right \rceil,
	\end{equation*}
	\eqref{3rdcond_col} does not hold.
	This results in 
	\begin{align}
			& \sum_{k \neq k^i}   \sum_{j \in \mathcal{O}_{k}^*} \sum_{t=t_k^\dag -1 }^{T} \mathds{1} \big\{ Q_i^k(t-1) \geq Q_j^k(t-1),  m_i < m_{b^k}, Mn_i^{\text{cent}}(t) \geq A \big\} \nonumber \\
			& \leq  \sum_{k \neq k^i}   \sum_{j \in \mathcal{O}_{k}^*}  \frac{2}{\lnn{1\!+\!\eta}}   \left( \frac{1}{(\gamma - 1)^2} +  {\frac{\gamma \lnn{1 + \epsilon_n)(1 + \eta)}}{\gamma - 1} +1} \right) \nonumber \\
				& \leq  \frac{M(M+1)}{\lnn{1\!+\!\eta}}   \left( \frac{1}{(\gamma - 1)^2} +  {\frac{\gamma \lnn{1 + \epsilon_n)(1 + \eta)}}{\gamma - 1} +1} \right). \label{eqn:Collisions_Selections_P1} 
	\end{align}

	We now examine the second part of \eqref{eqn:col_mainsplitting} when $m_i \geq m_{b^k}$ and split the conditional as
	\begin{align}
		\mathds{1} &\left\{ i^k(t) \! = \! i, m_i \! \geq \! m_{b^k}, Mn_i^{\text{cent}}(t) \geq A \right\} \nonumber \\
		& = \mathds{1} \left\{ i^k(t) \! = \! i, m_i \! \geq \! m_{b^k}, Mn_i^{\text{cent}}(t) \! \geq \! A , \mathcal{O}_{\omega^k}(t) = \mathcal{O}_{\omega^k}^* \right\} \nonumber \\
		&\quad + \mathds{1} \left\{ i^k(t) \! = \! i, m_i \! \geq \! m_{b^k}, Mn_i^{\text{cent}}(t) \! \geq \! A , \mathcal{O}_{\omega^k}(t) \neq \mathcal{O}_{\omega^k}^* \right \} \nonumber \\
		&\leq \mathds{1} \left\{ m_i \! \geq \! m_{b^k}, Mn_i^{\text{cent}}(t) \! \geq \! A ,  W_i^k(t \! - \! 1) \leq W_{b^k}^k(t \! - \! 1) \right \} \nonumber \\
		&\quad + \mathds{1} \left\{ m_i \! \geq \! m_{b^k}, Mn_i^{\text{cent}}(t) \! \geq \! A , W_i^k(t \! - \! 1) \leq W_{h}^k(t \! - \! 1) \right \} \label{eqn:secondary_splitting_pre}
	\end{align}
	for any arm $h\notin \mathcal{O}_{k}^*$. The two indicator functions in~\eqref{eqn:secondary_splitting_pre} can be combined as follows:
		\begin{equation*}	
		\eqref{eqn:secondary_splitting_pre} = \mathds{1} \left\{ m_i \! \geq \! m_{b^k}, Mn_i^{\text{cent}}(t) \! \geq \! A , W_i^k(t \! - \! 1) \leq W_{j}^k(t \! - \! 1) \right \},
	\end{equation*}
	for any $j\notin \mathcal{O}_{k}^* \setminus \{b^k\}$.
%
	This results in 
	\begin{align}
		& \sum_{k \neq k^i} \sum_{t=t_k^\dag}^{T} \mathds{1} \left\{ i^k(t) = i, m_i \geq m_{b^k}, Mn_i^{\text{cent}}(t) \geq A \right\} \nonumber \\
		& \qquad \leq \sum_{k \neq k^i} \sum_{j\notin  \mathcal{O}_{k}^* \setminus \{ b^k\}} \sum_{t=t_k^\dag}^{T} \mathds{1} \big\{ m_i \! \geq \! m_{b^k}, Mn_i^{\text{cent}}(t) \! \geq \! A , \nonumber \\
		& \qquad \qquad \qquad \qquad \qquad   W_i^k(t \! - \! 1) \leq W_{j}^k(t \! - \! 1) \big \}. \label{eqn:Collisions_Selections_P2}
	\end{align}
	 For $W_i^k(t) \leq W_{j}^k(t) $ to be true, at least one of the following must hold:
	\begin{align}
		\hat{\mu}_{i}(t) &\leq m_{i} - C_{i}^k(t) \label{1stcond_col2} \\
		\hat{\mu}_{j}(t) &\geq m_{j} + C_{j}^k(t) \label{2ndcond_col2} \\
		m_{i} &< m_{j} + 2 C_{j}^k(t). \label{3rdcond_col2} 
	\end{align}
	\eqref{1stcond_col2} and \eqref{2ndcond_col2} can be bounded using Theorem \ref{thm:EstDevBoudnsCondensed}. As before, \eqref{3rdcond_col2} never holds due to our choice of $A$.  Similarly to \eqref{eqn:Collisions_Selections_P1} 
	\begin{align}
	&\sum_{k \neq k^i} \sum_{t=1}^{T} \prob \left( i^k(t) = i, m_i \geq m_{b^k}, Mn_i^{\text{cent}}(t) \geq A \right) \nonumber \\
		& \leq \sum_{k \neq k^i} \sum_{j \notin \mathcal{O}_{k}^* \setminus \{b^k\}} \! \frac{2}{\lnn{1\!+\!\eta}}   \left( \frac{1}{(\gamma - 1)^2} +  {\frac{\gamma \lnn{1 + \epsilon_n)(1 + \eta)}}{\gamma - 1} +1} \right)  \nonumber\\
				& \leq \frac{2NM -M(M-1)}{\lnn{1\!+\!\eta}}   \left( \frac{1}{(\gamma - 1)^2} +  {\frac{\gamma \lnn{1 + \epsilon_n)(1 + \eta)}}{\gamma - 1} +1} \right).  \label{eqn:Collisions_Selections_P3} 	
	\end{align}
Using \eqref{eqn:col_mainsplitting}, \eqref{eqn:Collisions_Selections_P1}, and~\eqref{eqn:Collisions_Selections_P3} and accounting for the selections of arm $i$ during the initialization as in the proof of Theorem~\ref{thm:regret-coop-ucb2}, we obtain the bound in the theorem statement. 

\end{document}